\documentclass[12pt,twoside,final,a4paper]{amsart}
\usepackage{fullpage}
\usepackage{amsaddr}
\usepackage{amssymb}
\usepackage[all,cmtip]{xy}
\usepackage{cite}
\usepackage[utf8]{inputenc}
\usepackage{upgreek}
\usepackage{fix-cm}
\usepackage{graphicx}
\usepackage[T1]{fontenc}
\long\def\symbolfootnote[#1]#2{\begingroup%
\def\thefootnote{\fnsymbol{footnote}}\footnote[#1]{#2}\endgroup}

\newcommand{\tr}{\ensuremath{{}^T\!\!}}
\newcommand{\tra}{\ensuremath{{}^T}\!}
\newcommand{\C}{\mathfrak C}

\newcommand{\E}{\mathcal E}
\newcommand{\F}{\mathcal F}
\newcommand{\Aut}{\textup{Aut}}
\newcommand{\Hom}{\textup{Hom}}
\newcommand{\diag}{\textup{diag}}

\newtheorem{theorem}{Theorem}[section]
\newtheorem{lemma}[theorem]{Lemma}
\newtheorem{corollary}[theorem]{Corollary}

\newtheorem{definition}{Definition}[section]
\newtheorem*{theorem*}{Theorem}

\numberwithin{equation}{section}

\newcommand{\ignore}[1]{}

\newcommand{\mynote}[1]{}

\makeatletter
\def\Ddots{\mathinner{\mkern1mu\raise\p@
\vbox{\kern7\p@\hbox{.}}\mkern2mu
\raise4\p@\hbox{.}\mkern2mu\raise7\p@\hbox{.}\mkern1mu}}
\makeatother

\begin{document}
\title[MOR cryptosystem and Chevalley groups]{MOR Cryptosystem and classical Chevalley groups in odd characteristic}
\author[Mahalanobis and Singh]{Ayan Mahalanobis and Anupam Singh} 
\address{IISER Pune, Dr.~Homi Bhabha Road, Pashan, Pune 411008, INDIA.} 
\email{ayan.mahalanobis@gmail.com}
\email{anupamk18@gmail.com}
\thanks{This work was supported by a SERB research grant.}

\today

\begin{abstract}
In this paper we study the MOR cryptosystem with finite Chevalley groups. There are four infinite families of finite classical Chevalley groups. These are: special linear groups SL$(d,q)$, orthogonal groups O$(d,q)$ and symplectic groups Sp$(d,q)$. The family O$(d,q)$ splits to two different families of Chevalley groups depending on the parity of $d$. The MOR cryptosystem over SL$(d,q)$ was studied by the first author, ``A simple generalization of the ElGamal cryptosystem to non-abelian groups II, Communications in Algebra 40 (2012), no.~9, 3583–3596''. In that case, the hardness of the MOR cryptosystem was found to be equivalent to the discrete logarithm problem in $\mathbb{F}_{q^{d}}$. In this paper, we show that the MOR cryptosystem over Sp$(d,q)$ has the security of the discrete logarithm problem in $\mathbb{F}_{q^d}$. However, it seems likely that the security of the MOR cryptosystem for the family of orthogonal groups is $\mathbb{F}_{q^{d^2}}$. We also develop an analog of row-column operations in orthogonal and symplectic groups. 
\end{abstract}

\maketitle
\section{Introduction}
Public-key cryptography is a backbone of this modern society. However with \textbf{recent advances in the index-calculus algorithm to solve the discrete logarithm problem in finite fields of small characteristic} by Joux~\cite{joux,joux1}, and its possible implication to factoring algorithms, it seems that we are left with only one cryptographic primitive -- the discrete logarithm problem in the group of rational points of an elliptic curve over a finite field. So it seems prudent that we set out in search for new cryptographic primitives and subsequently new cryptosystems. The obvious question is: how to search and where to look? One can look into several well-known hard problems in Mathematics and hope to create a trap-door function or one can try to generalize the known, trusted cryptosystems.

This paper is in the direction of generalizing a known cryptosystem with the hope that something practical and useful will come out of this generalization. A new but arbitrary cryptosystem might not be considered by the community as a secure cryptosystem for decades. So our approach is conservative but practical.

The cryptosystem that we have in mind is \emph{the MOR cryptosystem}~\cite{crypto2001,Ma,Ma1}. It is a simple but powerful generalization of the well known and classic ElGamal cryptosystem. In this cryptosystem the discrete logarithm problem works in the automorphism group of a group instead of the group. As a matter of fact, it can work in the automorphism group of most algebraic structures. However, we will limit ourselves to finite groups. One way to look at the MOR cryptosystem is that it generalizes the discrete logarithm problem from a cyclic (sub)group to an arbitrary group.

The MOR cryptosystem over SL$(d,q)$ was studied earlier~\cite{Ma} and the work for finite $p$-groups is due to appear~\cite{Ma1}. It became clear that working with matrix groups of size $d$ over $\mathbb{F}_q$ and with automorphisms that act by conjugation, like the inner automorphism, there are two possible reductions of the security to finite fields. It is the security of the discrete logarithm problem in $\mathbb{F}_{q^d}$ or $\mathbb{F}_{q^{d^2}}$~\cite[Section 7]{Ma}.
This reduction is similar to the embedding of the discrete logarithm problem in the group of rational points of an elliptic curve to a finite field, the degree of the extension of that field over the field of definition of the elliptic curve is called the \emph{embedding degree}. In the case of SL$(d,q)$, it became the security of $\mathbb{F}_{q^{d}}$. The reason that we undertook this study, is to see, if the security in other classical Chevalley groups is $\mathbb{F}_{q^d}$ or $\mathbb{F}_{q^{d^2}}$. 

Though in cryptography it is often hard to come up with theorems about security of a cryptosystem, we were able to show that the attack that embeds the security of SL$(d,q)$ to a discrete logarithm problem in $\mathbb{F}_{q^d}$ works for symplectic groups as well. However, at this moment it seems likely that the security of the MOR cryptosystem in orthogonal groups O$(d,q)$ is $\mathbb{F}_{q^{d^2}}$. The way we implement this cryptosystem is by solving the word problem in generators. It presents no advantage to small characteristic. In the light of Joux's~\cite{joux} improvement of the index-calculus attack in small characteristic, this contribution of the MOR cryptosystem is remarkable.

\textbf{In summary}, the proposed MOR cryptosystem is totally different from the known ElGamal cryptosystems from a functional point of view. Its implementation depends on row-column operations and substitutions(substituting a matrix for a word in generators). However, we do have a concrete and tangible understanding of its security.
\subsection{Finite groups that we consider in this paper}
In this paper, we work with finite fields of odd characteristic. We work with classical Chevalley groups $\text{O}(2l+1,q), \text{Sp}(2l,q)$ and $\text{O}(2l,q)$ of $B_l, C_l$ and $D_l$ type respectively for $l\geq 2$. Note that the group $\text{SL}(l+1,q)$ of $A_l$ type has been dealt with earlier~\cite{Ma}. Our analysis and the subsequent conclusions hold for central (and subgroups of the center) quotients of the above mentioned groups and with any proper characteristic subgroups, like the commutator of the above groups. In this paper, we do not consider \emph{twisted classical Chevalley groups}, also called Steinberg groups. These are the ${}^2A_l(q)$ type which is the unitary group $\text{U}(l+1,q)$ and ${}^2D_l(q)$ type which is the orthogonal group $\text{O}^-(2l,q)$~\cite[Section 14.5]{ca}. We hope to continue our study with these groups in subsequent publications.

\subsection{Structure of the paper}
This paper is a study of the MOR cryptosystem using the orthogonal and symplectic groups over finite fields of odd characteristic.

In Section~\ref{des1}, we describe the MOR cryptosystem in some details. We emphasize that the MOR cryptosystem is a natural generalization of the classic ElGamal cryptosystem. In Section~\ref{des2}, we describe the orthogonal and symplectic groups and their automorphisms. In Section~\ref{wordproblem}, we describe two new algorithms. These algorithms use the row-column operations to write an element in the orthogonal or symplectic group as a word in generators. This is very similar to the row-column operations in special linear groups. These algorithms are useful in the implementation of the MOR cryptosystem. These algorithms are also of independent interest in computational group theory. We conclude this paper with some implementation details. 

\subsection{Notations and terminology}
It was bit hard for us to pick notations for this paper. The notations used by a Lie group theorist is somewhat different from that of a computational group theorist.We tried to preserve the essence of notations as much as possible. For example, a Lie group theorist will use SL$_{l+1}(q)$ to denote what we will denote by SL$(l+1,q)$ or SL$(d,q)$. We have used $\tr X$ to denote the transpose of the matrix $X$. This was necessary to avoid any confusion that might arise when using $X^{-1}$ and $\tr X$ simultaneously. In this paper, we use $K$ and $\mathbb{F}_q$ interchangeably, while each of them is a finite field of odd characteristic. All other notations used are standard.  
\section{The MOR Cryptosystem}\label{des1}
The MOR cryptosystem is a natural generalization of the classic
ElGamal cryptosystem. It was first proposed by Paeng
et.~al.~\cite{crypto2001}. To elaborate the idea behind a MOR
cryptosystem we take a slightly expository route. For the purpose of this exposition, we define \textbf{the discrete logarithm problem}. It is one of the most common cryptographic primitive in use. It works in any cyclic (sub)group $G=\langle g\rangle$, but is not secure in any cyclic group.
\begin{definition}[The discrete logarithm problem]
The discrete logarithm problem in $G=\langle g\rangle $ is, given $g$ and $g^m$ find $m$.
\end{definition}
The word ``find'' in the above definition is bit vague, in this paper we mean compute $m$. The hardness to solve the discrete logarithm problem depends on the presentation of the group and is not an invariant under isomorphism. It is believed that the discrete logarithm problem is secure in the multiplicative group of a finite field and the group of rational points of an elliptic curve. The security in elliptic curves is considered much better than that of finite fields because of non-existence of sub-exponential algorithms in most cases of elliptic curves~\cite{joseph,balu}.

A more important cryptographic primitive, related to the discrete logarithm problem is the \textbf{Diffie-Hellman problem}, also known as the \textbf{computational Diffie-Hellman problem}. 
\begin{definition}[Diffie-Hellman problem]
Given $g$, $g^{m_1}$ and $g^{m_2}$ find $g^{m_1m_2}$.
\end{definition}
It is clear, if one solves the discrete logarithm problem then the
Diffie-Hellman problem is solved as well. The other direction is not known.

The most prolific cryptosystem in use today is the ElGamal cryptosystem. It uses the cyclic group $G=\langle g\rangle$.
It is defined as follows:
\subsection{The ElGamal cryptosystem}\leavevmode\newline
A cyclic group $G=\langle g\rangle$ is public.
\begin{itemize}
\item \textbf{Public-key:} Let $g$ and $g^m$ is public.
\item \textbf{Private-key:} The integer $m$ is private.
\end{itemize}
\textbf{Encryption:}\leavevmode\newline
To encrypt a plaintext $\mathfrak{M}\in G$, get an arbitrary integer $r\in[1,|G|]$ compute $g^r$ and $g^{rm}$. 
The ciphertext is $\left(g^r,\mathfrak{M}g^{rm}\right)$.\newline
\textbf{Decryption:}\leavevmode\newline
After receiving the ciphertext $\left(g^r,\mathfrak{M}g^{rm}\right)$, the user uses the private key $m$. So she computes $g^{mr}$ from $g^r$ and them computes $\mathfrak{M}$.  

It is well known that the hardness of the ElGamal cryptosystem is
equivalent to the Diffie-Hellman problem~\cite[Proposition 2.10]{Intro}.
\subsection{The MOR cryptosystem}
In the case of the MOR cryptosystem, one works with the automorphism
group of a group. An automorphism group can be defined on any
algebraic structure and subsequently a MOR cryptosystem can also be
defined on that automorphism group, however in this paper we restrict
ourselves to finite groups. Furthermore, we look at \emph{classical groups} defined by generators and automorphisms are defined as actions on those generators.

Let $G=\langle g_1,g_2,\ldots,g_s\rangle$ be a finite group. Let $\phi$ be a non-identity automorphism.
\begin{itemize}
\item \textbf{Public-key:} Let $\{\phi(g_i)\}_{i=1}^s$ and $\{\phi^m(g_i)\}_{i=1}^s$ is public.
\item \textbf{Private-key:} The integer $m$ is private.
\end{itemize}
\textbf{Encryption:}\leavevmode\newline
To encrypt a plaintext $\mathfrak{M}\in G$, get an arbitrary integer $r\in[1,|\phi|]$ compute $\phi^r$ and $\phi^{rm}$. 
The ciphertext is $\left(\phi^r,\phi^{rm}\left(\mathfrak{M}\right)\right)$.\newline
\textbf{Decryption:}\leavevmode\newline
After receiving the ciphertext $\left(\phi^r,\phi^{rm}\left(\mathfrak{M}\right)\right)$, the user knows the private key $m$. So she computes $\phi^{mr}$ from $\phi^r$ and then computes $\mathfrak{M}$.
\begin{theorem}
The hardness to break the above MOR cryptosystem is equivalent to the Diffie-Hellman problem in the group $\langle\phi\rangle$.
\end{theorem}
\begin{proof}
It is easy to see that if one can break the Diffie-Hellman problem then one can compute $\phi^{mr}$ from $\phi^m$ in the public-key and $\phi^r$ in the ciphertext. This breaks the system.

On the other hand, observe that the plaintext is $\phi^{-mr}\left(\phi^{mr}(\mathfrak{M})\right)$. Assume that there is an oracle that can break the MOR cryptosystem, i.e., given $\phi,\phi^m$ and a plaintext $\left(\phi^r,g\right)$ will deliver $\phi^{-mr}(g)$. Now we query the oracle $s$ times with the public-key and the ciphertext $\left(\phi^r,g_i\right)$ for $i=1,2\ldots,s$. From the output one can easily find $\phi^{mr}(g_i)$ for $i=1,2,\ldots,s$. So we just witnessed that for $\phi^m$ and $\phi^r$ one can compute $\phi^{mr}$ using the oracle. This solves the Diffie-Hellman problem. 
\end{proof}
In a practical implementation of a MOR cryptosystem there are two things that matter the most.
\begin{description}
\item[a] The number of generators. As we saw that the automorphism $\phi$ is presented as action on generators. Larger the number of generators bigger is the public-key. 
\item[b] Efficient algorithm to solve the word problem. This means, given $G=\langle g_1,g_2,\ldots,g_s\rangle$ and $g\in G$, is there an efficient algorithm to write $g$ as word in $g_1,g_2,\ldots,g_s$? The reason of this importance is immediate -- the automorphisms are presented as action on generators and if one has to compute $\phi(g)$, then the word problem must be solved.
\end{description}
The obvious question is: what are the right groups for the MOR cryptosystem? In this paper, we pursue a study of the MOR cryptosystem using \textbf{finite Chevalley groups} of classical type, in particular, orthogonal and symplectic groups. 


\section{Classical Groups}\label{des2}

In this section, we produce a brief overview of the Chevalley groups of classical type.We introduce orthogonal and symplectic groups. References for this section are Carter~\cite{ca} and Grove~\cite{gr}. We also briefly describe similitude groups which are required for a study of diagonal automorphisms of the Chevalley groups. In this section we fix notation which will be used throughout this paper.

Let $V$ be a vector space of dimension $d$ over a field $K$ of odd characteristic. Let $\beta\colon V\times V \rightarrow K$ be a bilinear form. By fixing a basis of $V$ we can associate a matrix to $\beta$. We shall abuse the notation slightly and denote the matrix of the bilinear form by $\beta$ itself. Thus $\beta(x,y)=\tr x\beta y$ where $x,y$ are column vectors. We will work with non-degenerate bilinear forms and that means $\det\beta\neq 0$. A symmetric or skew-symmetric bilinear form $\beta$ satisfies $\beta=\tr\beta$ or $\beta=-\tr\beta$ respectively.  
\begin{definition}[Orthogonal Groups]
A square matrix $X$ of size $d$ is called orthogonal if $\tr X\beta X=\beta$ where $\beta$ is symmetric. It is well known that the orthogonal matrices form a group known as the orthogonal group.
\end{definition}
\begin{definition}[Symplectic Group]
A square matrix of size $d$ is called symplectic if $\tr X\beta X=\beta$ where $\beta$ is skew-symmetric. And the set of symplectic matrices form symplectic group.
\end{definition}
We write the dimension of $V$ as $d=2l+1$ or $d=2l$ for $l\geq 1$. 
We fix a basis and index it by $0,1,2,\ldots,l,-1,-2,\ldots,-l$ for odd dimension and by $1,2,\ldots,l,-1,-2,\ldots,-l$ for even dimension. We consider the non-degenerate bilinear
forms $\beta$ on $V$ given by the following matrices: 
\begin{itemize}\label{forms}
\item Type $B_l$: The form $\beta$ is symmetric with $d=2l+1$ and $\beta=\begin{pmatrix}2&0&0\\ 0&0&I_l\\ 0&I_l&0\end{pmatrix}$.
\item Type $C_l$: The form $\beta$ is skew-symmetric with $d=2l$ and  $\beta=\begin{pmatrix}0&I_l\\ -I_l&0\end{pmatrix}$.
\item Type $D_l$: The form $\beta$ is symmetric with $d=2l$ and  $\beta=\begin{pmatrix}0&I_l\\ I_l&0\end{pmatrix}$.
\end{itemize}
where $I_l$ is the identity matrix of size $l$ over $K$. 

Let $K(=\mathbb{F}_q)$ be a finite field of odd characteristic. If $d$ is odd there is only one orthogonal group up to conjugation~\cite[Page79]{gr} and thus we can fix $\beta$ as above of $B_l$ type. In this case the orthogonal group is simply denoted by $\text{O}(2l+1,q)$. 
Up to equivalence there is only one non-degenerate skew-symmetric form in even dimension~\cite[Theorem 2.10]{gr}. We fix $\beta$ of $C_l$ type as above. Thus there is only one symplectic group up to conjugation denoted by $\text{Sp}(2l,q)$.
However up to conjugation there are two different orthogonal groups~\cite[Page 79]{gr} in even dimension $d=2l$. In this paper, we work with only one of them corresponding to the $\beta$ fixed as above of type $D_l$. We denote this orthogonal group by $\text{O}(2l,q)$. The other orthogonal group often denoted as $\text{O}^-(2l,q)$ is twisted Chevalley group denoted as ${}^{2}D_l(q)$, also called Steinberg groups. 
\begin{definition}[Orthogonal similitude groups]
The orthogonal similitude group is defined as the set of matrices $X$ of size $d$ as follows: $\text{GO}(d,q)=\{X\in \text{GL}(d,q)\mid \tr X\beta X=\mu \beta, \mu\in\mathbb{F}_q^\times\}$ where $d=2l+1$ or $2l$ and $\beta$ is of type $B_l$ and $D_l$ respectively. 
\end{definition}
\begin{definition}[Symplectic similitude group]
The symplectic similitude group is denoted by $\text{GSp}(2l,q)=\{X\in \text{GL}(2l,K)\mid \tr X \beta X=\mu \beta, \mu\in \mathbb{F}_q^\times\}$
where $\beta$ is of type $C_l$. 
\end{definition}
Here $\mu$ depends on the matrix $X$ and is called the similitude factor. 
The similitude factor $\mu$ defines a group homomorphism from the similitude group to $\mathbb{F}_q^\times$ and the kernel is the orthogonal group $\text{O}(d,q)$ when $\beta$ is symmetric and symplectic group $\text{Sp}(2l,q)$ when $\beta$ is skew-symmetric respectively~\cite[Section 12]{kmr}. 
Note that scalar matrices $\lambda I$ for $\lambda \in \mathbb{F}_q^\times$ belong to the center of similitude groups. The similitude groups are thought of analog of what $\text{GL}(d,q)$ is for $\text{SL}(d,q)$. For a discussion of the diagonal automorphisms of Chevalley groups we need the diagonal subgroups of the similitude groups. 
\begin{definition}[Diagonal group]\label{diagonalgroup} 
The diagonal groups are defined to be the group of non-singular diagonal matrices in the corresponding similitude group and are as follows:
in the case of $\text{GO}(2l+1,q)$ it is 
$$\{\diag(\alpha, \lambda_1,\cdots,\lambda_l,\mu\lambda_1^{-1},\cdots,\mu\lambda_l^{-1})\mid \lambda_1,\ldots,\lambda_l, \alpha^2=\mu\in \mathbb{F}_q^\times\}$$
and in the case of $\text{GO}(2l,q)$ and $\text{GSp}(2l,q)$ it is 
$$\{\diag(\lambda_1,\cdots,\lambda_l,\mu\lambda_1^{-1},\cdots,\mu\lambda_l^{-1})\mid \lambda_1,\ldots,\lambda_l,\mu\in \mathbb{F}_q^\times\}.$$
\end{definition}
Conjugation by these diagonal elements produce diagonal automorphisms in the respective Chevalley groups. 

We denote by $\Omega_d(q)$ the commutator subgroup of the orthogonal group O$(d,q)$. It is  a index 2 subgroup of the special orthogonal group $\text{SO}(d,q)$. We fix a generator of $\text{SO}(d,q)/\Omega_d(q)$ as  $d(\zeta)=\diag(1,1,\ldots,1,\zeta,1,\ldots,1,\zeta^{-1})$ where $\zeta$ is a fixed non-square in $\mathbb{F}_q$~\cite[Theorem 9.7]{gr}. Further, the group $\text{SO}(d,q)$ is of index $2$ in $\text{O}(d,q)$ and we fix a generator for the quotient as $w_l=I-e_{l,l}-e_{-l,-l}-e_{l,-l}-e_{-l,l}$ where $e_{i,j}$ denotes a matrix with $1$ at $(i,j)\textsuperscript{th}$ place and $0$ everywhere else.

\subsection{Chevalley Generators}\label{chgenerators}

To work with Chevalley groups we need a set of generators for these groups. We describe the Chevalley generators from the theory of Chevalley groups~\cite{ca}. For sake of completeness of this paper, we will briefly go through the theory of Chevalley groups in the next section. In what follows $t$ varies over $\mathbb{F}_q$.
\begin{enumerate}
\item The group $\text{SL}(l+1,q)$ is generated by the matrices $x_{i,j}(t)=I+te_{i,j}$ where $1\leq i\neq j\leq l+1$. This is Chevalley group of $A_l$ type. 
\item For $1\leq i,j \leq l$, the group $\Omega_{2l+1}(q)$ is generated by the following matrices: 
\begin{eqnarray*}
x_{i,j}(t)=&I+t(e_{i,j}-e_{-j,-i})& \text{for}\; i\neq j,\\
x_{i,-j}(t)=&I+t(e_{i,-j}-e_{j,-i}),&\text{for}\; i<j,\\
x_{-i,j}(t)=&I+t(e_{-i,j}-e_{-j,i})&\text{for}\; i<j ,\\ x_{i,0}(t)=&I+t(2e_{i,0}-e_{0,-i})-t^2e_{i,-i},&\\
x_{0,i}(t)=&I+t(-2e_{-i,0}+e_{0,i})-t^2e_{-i,i}.&
\end{eqnarray*}
\noindent With these generators the elements $d(\zeta)=\diag(1,\underbrace{1,\cdots,1,\zeta}_l, \underbrace{1,\cdots,1,\zeta^{-1}}_l)$ and $w_l=I-e_{l,l}-e_{-l,-l}-e_{l,-l}-e_{-l,l}$ generate the orthogonal group $\text{O}(2l+1,q)$. This is Chevalley group of $B_l$ type.
\item For $1\leq i,j\leq l$, the group $\text{Sp}(2l,q)$ is generated by the matrices 
\begin{eqnarray*}
x_{i,j}(t)=&I+t(e_{i,j}-e_{-j,-i}) &\text{for}\; i\neq j,\\ 
x_{i,-j}(t)=&I+t(e_{i,-j}+e_{j,-i})&\text{for}\; i<j,\\ 
x_{-i,j}(t)=&I+t(e_{-i,j}+e_{-j,i})& \text{for}\; i<j,\\
x_{i,-i}(t)=&I+te_{i,-i}&\\
x_{-i,i}(t)=&I+te_{-i,i}.&
\end{eqnarray*}
  This is Chevalley group of $C_l$ type.
\item For $1\leq i,j\leq l$, the group $\Omega_{2l}(q)$ is generated by the matrices
\begin{eqnarray*}
x_{i,j}(t)=&I+t(e_{i,j}-e_{-j,-i}) &\text{for}\; i\neq j,\\
x_{i,-j}(t)=&I+t(e_{i,-j}-e_{j,-i}) &\text{for}\;i<j,\\
x_{-i,j}(t)=&I+t(e_{-i,j}-e_{-j,i})&\text{for}\; i<j.
\end{eqnarray*}
With the above generators the elements $d(\zeta)=\diag(\underbrace{1,\ldots,1,\zeta}_l,\underbrace{1,\ldots,1,\zeta^{-1}}_l)$ and $w_l=I-e_{l,l}-e_{-l,-l}-e_{l,-l}-e_{-l,l}$ generate the orthogonal group $\text{O}(2l,q)$.  This is Chevalley group of $D_l$ type.
\end{enumerate}
It is interesting to note that our algorithm in Section~\ref{wordproblem} to solve the word problem in Chevalley groups using the above generators gives yet another proof that the matrices listed above generate the corresponding groups. 


\section{Adjoint Chevalley Groups}\label{chevalleygroups}
In this section we introduce adjoint Chevalley groups. One could get around without reading this section, we include this to explain why the generators listed in Section~\ref{chgenerators} are natural. The material is not that important to understand the later part of this paper. It is probably impossible to produce a brief and comprehensive introduction to Chevalley groups. The usual route to describe a Chevalley group is as a subgroup of the automorphism group of a \emph{simple Lie algebra}. A Lie algebra over a field is a finite dimensional vector space with a Lie bracket operation. We are particularly interested in simple Lie algebras over $\mathbb{C}$. 
The theory was originally developed by Chevalley~\cite{ch}. Though our exposition follows Carter~\cite{ca} and Steinberg~\cite{st2}. A more general account of this theory over commutative rings can be found in Vavilov~\cite{va}. 

We are particularly interested in simple Lie algebras over $\mathbb{C}$. One dimensional Lie algebras are always simple and uninteresting. It is known that a Lie algebra contains a self-normalizing nilpotent subalgebra $\mathcal{H}$ called the Cartan subalgebra. In case of simple Lie algebras the Cartan subalgebra contains only semi-simple elements. Corresponding to a Cartan subalgebra $\mathcal{H}$, we can define a decomposition of the simple Lie algebra $\mathcal{L}$ by looking at the simultaneous decomposition as eigen-spaces. So we can write $\mathcal{L}=\mathcal{H}\bigoplus\oplus_{r\in \Phi}\mathcal{L}_r$ where $\mathcal{L}_r$ are one dimensional subspaces of $\mathcal{L}$ satisfying $[h,e_r]=r(h)e_r$ where $r\colon \mathcal{H}\rightarrow\mathbb{C}$, $e_r$ is the generator of $\mathcal{L}_r$ and $\Phi$ is finite subset of $\mathcal{H}^\ast$, the dual of $\mathcal{H}$. This is called a Cartan decomposition of $\mathcal L$~\cite[Section 3.2]{ca}.

The set $\Phi$ obtained using Cartan decomposition is the  root system for the Lie algebra $\mathcal{L}$. An abstract root system~\cite[Definition 2.1.1]{ca} is a finite subset of an Euclidean space of the same dimension as $\mathcal H$. By fixing an order in the Euclidean space we get a system of positive roots $\Phi^+$ and negative roots $\Phi^{-}$ so that $\Phi=\Phi^+\cup \Phi^-$. Let $\Pi=\{p_1,\ldots,p_l\}$ be a system of \emph{simple roots}, i.e., any root is either non-positive or non-negative integer linear combination of simple roots. We denote by
$h_r=2r/(r,r)$ the co-root corresponding to the root $r$, where $(.\,,.)$ is the usual inner-product on the Euclidean space containing the roots. It is a theorem of Chevalley that there is a basis 
$\{h_r,r\in \Pi; e_r,r\in\Phi\}$ of $\mathcal L$  satisfying the following~\cite[Theorem 4.2.1]{ca}:
\begin{enumerate}
\item[] $[e_r,e_{-r}]=h_r$,
\item[]  $[e_r,e_s]=0$ if $r+s\not\in \Phi$ else $[e_r,e_s]=\pm(p+1)e_{r+s}$ where $r\neq \pm s$,
\item[] $[h_r,h_s]=0$,
\item[] $[h_r,e_s]=A_{rs}e_s$,
\end{enumerate}
where $-pr+s,\ldots,-r+s, s, r+s,\ldots, qr+s$ is a $r$-chain passing through $s$ and
$A_{rs}=\frac{2(r,s)}{(r,r)}$ are integers known as the \emph{Cartan integers}. Such a basis is called a \emph{Chevalley basis}~\cite[Section 4.2]{ca}. 

There is a well-known classification of finite dimensional simple Lie algebras over $\mathbb C$~\cite[Section 3.6]{ca}. They are classified via their Dynkin diagram. There are four infinite families $A_l (l\geq 1)$, $B_l (l\geq 2)$, $C_l (l\geq 3)$ and $D_l (l\geq 4)$ together called simple Lie algebras of "classical type" and five "exceptional types" $G_2, F_4, E_6, E_7$ and $E_8$. In Section~\ref{classicaltype} we explicitly describe the classical Lie algebras and their  Chevalley basis which will be used to form adjoint Chevalley groups. From now on $\mathcal L$ is one of $A_l, B_l, C_l$ or $D_l$.

Let $K(=\mathbb{F}_q)$ be a finite field of odd characteristic. We denote by $\mathcal L_{\mathbb Z}$  the $\mathbb Z$-span of a Chevalley basis in $\mathcal L$. 
Clearly $\mathcal L_{\mathbb Z}$ is a Lie
algebra over $\mathbb Z$. Define $\mathcal L_K = K\otimes \mathcal L_{\mathbb Z}$. 
Then one can define a  Lie algebra structure on $\mathcal L_K$ as follows: 
$$[1\otimes x, 1\otimes y]:=1\otimes[x,y]$$
for basis elements $x,y$ and extended by linearity. Thus $\mathcal L_K$ is a Lie algebra over $K$.

To define the groups of our interest we need to work with certain operators which are in $\Aut(\mathcal L_K)$. 
For this we start by defining $x_r(\zeta):=\exp(\zeta \textrm{ad}(e_r)) \in \Aut(\mathcal L)$ for $r\in \Phi$ and $\zeta \in \mathbb C$. Where ad is the Lie algebra homomorphism $\text{ad}\colon \mathcal{L}\rightarrow \text{End}\left(\mathcal{L}\right)$ given by $\text{ad}(x).y=[x,y]$. 
These operators $x_r$ are unipotent operators whose matrix entries are polynomials in $\zeta$ with integer coefficients. Thus by substituting $t\in K$ for the variable $\zeta$ and reducing the coefficients modulo the characteristic of the field $K$, we get operators $x_r(t)\in \Aut(\mathcal L_K)$. 
The {\it adjoint Chevalley group of type $\mathcal L$ over $K$} is the subgroup of $\Aut(\mathcal L_K)$ generated by $x_r(t)$ for all 
$r\in \Phi$ and $t\in K$, and is denoted by
\[\mathcal L(K):=\left\langle x_r(t) \mid r\in\Phi, t\in K \right\rangle.\]
One can explicitly write down $x_r(t)$ as an automorphism of $\mathcal L_K$ on the basis elements as follows:
\begin{enumerate}
\item[] $x_r(t).e_r=e_r$,
\item[] $x_r(t).e_{-r}= e_{-r}+th_r-t^2e_r$,
\item[] $x_r(t).h_s=h_s - A_{sr}te_r$ for $s\in \Pi$,
\item[] $x_r(t).e_s = \sum_{i=0}^q M_{r,s,i}t^ie_{i,r+s}$ if $r\neq \pm s$ 
\end{enumerate}
where $M_{r,s,i}=\pm {p+i\choose i}$ and $r,s\in \Phi$. 
In this paper we are working with classical Chevalley groups which are explicitly described in Section~\ref{classicaltype}. 

For a fixed $r$, the subgroup $X_r$ generated by elements $x_r(t)$ for all $t\in K$, 
is called a root subgroup and is isomorphic to the additive group of $K$.
Let $U:=\left\langle X_r \mid r\in \Phi^+\right\rangle$ and $V:=\left\langle X_r \mid r\in \Phi^-\right\rangle$ be subgroups 
of $\mathcal{L}(K)$. Then both $U$ and $V$ are unipotent as well as nilpotent groups. Furthermore these are
Sylow $p$-subgroups of $\mathcal{L}(K)$.  For every $r\in \Phi$ there is a surjective homomorphism~\cite[Theorem 6.3.1]{ca} 
$\phi_r\colon \textrm{SL}(2,K) \rightarrow \left\langle X_r, X_{-r}\right\rangle$ which maps 
$\begin{pmatrix} 1 & t \\ 0 & 1\end{pmatrix}$ to $x_r(t)$ and
$\begin{pmatrix} 1 & 0 \\ t & 1\end{pmatrix}$ to $x_{-r}(t)$.
Let us define $h_r(\lambda)$ as $\phi_r \begin{pmatrix} \lambda & 0 \\ 0 & \lambda^{-1}\end{pmatrix} $ and 
$n_r(t)$ as $\phi_r \begin{pmatrix} 0 & t \\ -t^{-1} & 0\end{pmatrix}$.
Set $n_r=n_r(1)$ for convenience. We now define some important subgroups $H:=\left\langle h_r(t) \mid r\in \Phi, t \in K^\times\right\rangle$ and $N:=\left\langle H, n_r \mid r\in \Phi\right\rangle$.

For our discussion of diagonal automorphisms we need a slightly larger group.
Let $P=\mathbb Z\Phi$ be the root lattice and $Q$ be the weight lattice ($\mathbb Z$ span of the dual of co-roots)~\cite[Section 7.1]{ca}. We know that $P\subset Q$. It is known that: 
\begin{center}\label{lattice}
\begin{tabular}{c|c}
\hline
Simple Lie Algebra & $Q/P$\\
\hline
 $A_l$ & $\mathbb Z/(l+1)\mathbb Z$ \\
 $B_l, C_l$ &   $\mathbb Z/2\mathbb Z$ \\
 $D_l$ & $\mathbb Z/4\mathbb Z$ if $l$ odd\\
 & $\mathbb Z/2\mathbb Z \times \mathbb Z/2\mathbb Z$ if $l$ even \\
 \hline 
\end{tabular}
\end{center}
There is a well known isomorphism~\cite[Section 7.1]{ca} $h\colon \Hom(P, K^\times) \rightarrow \hat H\subset \Aut(\mathcal L_K)$ 
given by $\chi\mapsto h(\chi)$ where $h(\chi).h_s=h_s$ and $h(\chi).e_s=\chi(s)e_s$. Furthermore
$H\subset \hat H$ and $h(\chi)\in H$ if the character $\chi$ can be extended to a character
of $Q$. The group $\hat H$ normalizes $U$ and $V$ and hence $\mathcal L(K)$ (refer to the note following~\cite[Theorem 7.1.1]{ca}). 
Further $\mathcal L(K) \cap \hat H =H$. Let
$\hat G \subset \Aut(\mathcal L_K)$ be the subgroup generated by $\mathcal L(K)$ and $\hat H$. Then
$\mathcal L(K)$ is a normal subgroup of $\hat G$ and $\hat G/\mathcal L(K) \cong \hat H/H$.

We are working with Chevalley groups of classical type which we now describe explicitly.  
\subsection{Chevalley Groups of Classical types}\label{classicaltype}
In this section, we describe the Chevalley groups of classical type following Carter ~\cite[Section 11.2 and 11.3]{ca}.
In each case, we first describe the complex simple Lie algebra. These are subalgebra of the full matrix algebra $M(d,\mathbb C)$, square matrices of size $d$ over $\mathbb{C}$, with bracket operation $[X,Y]=XY-YX$. Then we get a Chevalley basis as described earlier and a Lie algebra over $\mathbb Z$ and hence over any field $K$ by a base change. Using this we describe the root generators $x_r(t)$ of the adjoint Chevalley group $\mathcal L(K)$. 
It turns out that the operators $x_r(t)$ are inner conjugation automorphism~\cite[Lemma 4.5.1]{ca} on $\mathcal L_K$ by $\exp(te_r)$ which generate an intermediate Chevalley group denoted as $\bar G$. The group $\bar G$ is close to groups of our interest. In later section, we will abuse the notation slightly and denote the generators of $\bar G$ as $x_r(t)$ (for example in the Section~\ref{chgenerators}).
We make a table before we describe them explicitly.
\begin{center}
\begin{tabular}{c|c|c|c}
\hline\\
 Type & Group of our interest & $\bar G$ & $\mathcal L(K)$ \\
\hline
 $A_l$ & $\text{SL}(l+1,K)$ &$\text{SL}(l+1,K)$&$\text{PSL}(l+1,K)$ \\
 $B_l$ &   $\text{O}(2l+1,K)$ & $\Omega_{2l+1}(K)$ & $P\Omega_{2l+1}(K)$\\
 $C_l$ & $\text{Sp}(2l,K)$ & $\text{Sp}(2l,K)$& $\text{PSp}(2l,K)$ \\
 $D_l$ &   $\text{O}(2l,K)$ & $\Omega_{2l}(K)$ & $P\Omega_{2l}(K)$\\
 \hline 
\end{tabular}
\end{center}

{\bf Type $A_l$ :}   The $A_l$ type complex Lie algebra is $sl_{l+1}(\mathbb C)$ consisting of trace $0$ matrices of size $l+1$. The set of all diagonal matrices in $sl_{l+1}(\mathbb C)$ give a Cartan subalgebra and that Cartan decomposition gives a Chevalley basis. The roots (eigen-vectors for non-zero eigen-values) which are part of Chevalley basis is given by $\Phi=\{e_{i,j}\mid 1\leq i\neq j\leq l+1\}$. We fix a simple root system $\Pi=\{e_{i,i+1} \mid 1\leq i\leq l\}$. A Chevalley basis is obtained by taking union of $\Phi$ with the set $\{[e_{i,i+1},e_{i+1,i}]\mid i\leq i\leq l\}$.

Thus the generators for the intermediate Chevalley group of type $A_l$ over field $K$ are $x_{i,j}(t)=I+te_{i,j}$ where $i\neq j$ and $t\in K$. Hence $\bar G = \text{SL}(l+1,K)$ and the adjoint group is $A_l(K)\cong \text{PSL}(l+1,K)$. 

{\bf Type $B_l$ :}  The $B_l$ type complex Lie algebra is $o_{2l+1}(\mathbb C)=\{X\in M(2l+1,\mathbb C)\mid \tra X\beta +\beta X=0\}$ where $\beta$ is as in the Section~\ref{forms}. Any $X\in o_{2l+1}(\mathbb C)$ is of the form $\begin{pmatrix} 0& X_{01} & X_{02} \\ -2\tra X_{02} & X_{11} & X_{12} \\-2\tra X_{01} & X_{21}& -\tra X_{11} \end{pmatrix} $ where $X_{12}$
and $X_{21}$ are skew-symmetric matrices of size $l\times l$. 
The set of diagonal matrices give a Cartan subalgebra and the Cartan decomposition gives us a Chevalley basis. Thus the roots in this case are $\Phi=\{e_{i,j}-e_{-j,-i}, -e_{-i,-j}+e_{j,i},
e_{i,-j}-e_{j,-i}, e_{-i,j}-e_{-j,i}, 2e_{i0}-e_{0,-i}, -2e_{-i,0}+e_{0,i} \mid 1\leq i <j \leq l\}$. 
The simple roots are $\Pi=\{e_{i,i+1}-e_{-(i+1),-i}, 2e_{l,0}-e_{0,-l} \mid 1\leq i\leq l-1\}$.

In this case the intermediate Chevalley group is $\Omega_{2l+1}(K)$ generated by the Chevalley generators: For $1\leq i,j\leq l$, 
\begin{eqnarray*}
x_{i,j}(t)=& I+t(e_{i,j}-e_{-j,-i})&\text{for}\; i\neq j,\\ 
x_{i,-j}(t)=&I+t(e_{i,-j}-e_{j,-i}) & \text{for}\; i<j,\\ 
x_{-i,j}(t)=&I+t(-e_{-i,j}+e_{-j,i})& \text{for}\; i<j,\\
x_{i,0}(t)=&I+t(2e_{i,0}-e_{0,-i})-t^2e_{i,-i},\\ x_{0,i}(t)=&I+t(-2e_{-i,0}+e_{0,i})-t^2e_{-i,i}.
\end{eqnarray*} 
The adjoint group is  $B_l(K) \cong P\Omega_{2l+1}(K)$.

{\bf Type $C_l$ :}  The complex Lie algebra of type $C_l$ is $sp_{2l}(\mathbb C)=\{X\in M(2l,\mathbb C)\mid \tra X\beta +\beta X=0\}$ where $\beta $ is as in Section~\ref{forms}. Any $X\in sp_{2l}(\mathbb C)$ is of the form $\begin{pmatrix} X_{11} & X_{12} \\ X_{21} & -\tra X_{11}\end{pmatrix} $ where $X_{12}$
and $X_{21}$ are symmetric matrices. 
The set of diagonal matrices is a Cartan subalgebra and the Cartan decomposition gives a Chevalley basis. The roots in this case are
 $\{e_{i,j}-e_{-j,-i}, -e_{-i,-j}+e_{j,i},
e_{i,-j}+e_{j,-i}, e_{-i,j}+e_{-j,i}, e_{i,-i}, e_{-i,i} \mid 1\leq i <j \leq l\}$. 
The simple roots are
$\Pi=\{e_{i,(i+1)}-e_{-(i+1),-i}, e_{l,-l} \mid 1\leq i\leq l-1\}$. 

The root generators for the group over a field $K$ are: For $1\leq i,j\leq l$  
\begin{eqnarray*}
x_{i,j}(t)=&I+t(e_{i,j}-e_{-j,-i})&\text{for}\; i\neq j,\\ x_{i,-j}(t)=&I+t(e_{i,-j}+e_{j,-i}) & \text{for}\; i<j,\\
x_{-i,j}(t)=&I+t(e_{-i,j}+e_{-j,i})& \text{for}\; i<j,\\ 
x_{i,-i}(t)=&I+te_{i,-i},&\\
x_{-i,i}(t)=&I+te_{-i,i}& 
\end{eqnarray*}
which generate the intermediate Chevalley group $\text{Sp}(2l,K)$. The adjoint group is $C_l(K) \cong\text{PSp}(2l,K)$.

{\bf Type $D_l$ :} The $D_l$ type complex Lie algebra is $o_{2l}(\mathbb C)=\{X\in M_{2l}(\mathbb C)\mid \tra X\beta +\beta X=0\}$ where $\beta$ is
as in Section ~\ref{forms}. Any $X\in o_{2l}(\mathbb C)$ is of the form 
$\begin{pmatrix} X_{11} & X_{12} \\ X_{21} & -\tra X_{11}\end{pmatrix} $ where $X_{12}$
and $X_{21}$ are skew-symmetric matrices. The set of diagonal matrices form a Cartan subalgebra. The roots in this case are $\{e_{i,j}-e_{-j,-i}, -e_{-i,-j}+e_{j,i},
e_{i,-j}-e_{j,-i}, e_{-i,j}-e_{-j,i} \mid 1\leq i <j \leq l\}$. The simple roots are
$\Pi=\{e_{i,(i+1)}-e_{-(i+1),-i}, p_l=e_{(l-1),-l}-e_{l,-(l-1)} \mid 1\leq i\leq l-1\}$. This gives us a Chevalley basis.

The intermediate Chevalley group in this case is $\Omega_{2l}(K)$ generated by the Chevalley generators: For $1\leq i,j\leq l$ 
\begin{eqnarray*}
x_{i,j}(t)=&I+t(e_{i,j}-e_{-j,-i})& \text{for}\; i\neq j,\\ x_{i,-j}(t)=&I+t(e_{i,-j}-e_{j,-i})&\text{for}\; i<j, \\
x_{-i,j}(t)=&I+t(e_{-i,j}-e_{-j,i})& \text{for}\; i<j. 
\end{eqnarray*}
The adjoint group is $D_l(K) \cong P\Omega_{2l}(K)$.

\section{Description of Automorphism Group of Classical Groups}\label{automorphismclassical}

To build a MOR cryptosystem we need to work with the automorphism group of Chevalley groups. In this section we describe the automorphism group of classical groups following Dieudonne~\cite{di}. 
Let $G$ be one of the following groups: adjoint or intermediate Chevalley group of classical type or more generally the groups listed in the table in section~\ref{classicaltype}.

{\bf Conjugation Automorphisms}: For $t\in G$ the map given by $g\mapsto tgt^{-1}$ is an 
automorphism of $G$, called an \textbf{inner automorphism}. More generally if $G$ is a normal subgroup of $N$ then the conjugation maps $g\mapsto ngn^{-1}$ for $n\in N$ are called conjugation automorphisms of $G$.

{\bf Central Automorphisms}: Let $\chi\colon G\rightarrow \mathcal Z(G)$ be a group homomorphism to the center of the group. Then the map $g\mapsto \chi(g)g$ is an automorphism of $G$, known as the central automorphism. There are no non-trivial central automorphisms for perfect groups, for example, the adjoint Chevalley groups $\text{SL}(l+1,K)$ and $\text{Sp}(2l,K)$, $K\geq 4$ and $l\geq 2$. In case of orthogonal group, the center is of two elements $\{I,-I\}$. Any map $\chi$ maps $\Omega_d(K)$ to identity. This implies that there are at most four central automorphisms in this case.

{\bf Field Automorphisms}: Let $f\in \Aut(K)$. Then the map $x_r(t)\mapsto x_r(f(t))$ for all 
$r\in\Phi$ and $t\in K$ extends to an automorphism of $G$. These are called field automorphism. In terms of matrices these amount to
replacing each term of the matrix by its image under $f$.

{\bf Graph Automorphisms}:  A symmetry of Dynkin diagram induces such automorphisms. This
way we get automorphisms of order $2$ for $A_l(K), l\geq 2$ and $D_l(K), l\geq 4$. 
We also get an automorphisms of order $3$ for $D_4(K)$. This map is given by $x_r(t)\mapsto x_{\bar r}(\gamma_rt)$ where $r\mapsto \bar r$ is Dynkin
diagram automorphism and $\gamma_r=\pm 1$.

In the case of $A_l$ for $l\geq 2$, the map $x\mapsto A^{-1}\tr x^{-1}A$ where 
\[A= \begin{pmatrix} 
0&\cdots&0&0&0 & 1\\ 
0&\cdots&0&0& -1&0\\
0&\cdots&0&1&0&0\\
0&\cdots&-1&0&0&0\\
\vdots& \Ddots &\vdots&\vdots&\vdots&\vdots\\
(-1)^{l-1}&\cdots&0&0&0&0
\end{pmatrix} \]
explicitly describes the graph automorphism. 

In the case of $D_l$ for $l\geq 5$, the graph automorphism is given by $x\mapsto B^{-1}xB$
where $B$ is a permutation matrix obtained from identity matrix of size $2l\times 2l$ by switching
the $l\textsuperscript{th}$ row and $-l\textsuperscript{th}$ row. 
This automorphism is a conjugating automorphism.

\begin{theorem}[Dieudonne]\label{autoch} 
Let $K$ be a field of odd characteristic and $l\geq 2$.
 \begin{enumerate}
  \item For the group $\textrm{SL}(l+1,K)$ any automorphism is of the form  $\iota \gamma \theta$ where $\iota$ is a conjugation automorphism defined by elements of $\text{GL}(l+1,K)$ and $\gamma$ is a graph automorphism of $A_l$ type. 
\item For the group $\text{O}(d,K)$ any automorphism is of the form $c_{\chi} \iota \theta$ where $c_{\chi}$ is a central automorphism, $\iota$ a conjugation automorphism by $\text{GO}(d,K)$ elements (this includes the graph automorphism of $D_l$ case). 
\item For the group $\text{Sp}(2l,K)$ any automorphism is of the form $\iota \theta$ where $\iota$ is a conjugation automorphism by $\text{GSp}(2l,K)$ elements.
\end{enumerate}
In all cases $\theta$ denotes field automorphisms.
\end{theorem}
In the above theorem, conjugation automorphisms are given by conjugation by elements of a larger group ans it includes the group of inner automorphisms. We introduce diagonal automorphisms to make it more precise. The conjugation automorphisms $\iota$ can be written as a product of $\iota_g$ and $\delta$ where $\iota_g$ is an inner automorphism and $\delta$ is a diagonal automorphism. 
 
 {\bf Diagonal Automorphisms}: The adjoint Chevalley group $\mathcal{L}(K)$ is normalized by $\hat H$ which is a subgroup of 
$\Aut(\mathcal L_K)$. Thus for $h(\chi)\in \hat H$ which is not in $H$ gives an automorphism
$g\rightarrow h(\chi)gh(\chi)^{-1}$ (which is not an inner automorphism). Such automorphisms 
are called diagonal automorphism. The explicit action on generators is as follows:
$h(\chi)x_r(t)h(\chi)^{-1}= x_r(\chi(r)t)$. The group $\hat G$ is identified in~\cite[Chapter III, Section 6]{sel} with corresponding similitude group.
In the case of $A_l$ the diagonal automorphisms are given by conjugation by diagonal
elements of $\text{PGL}(l+1,q)$ on $A_l(q)=\text{PSL}(l+1,q)$. 
In the case of $B_l, C_l$ and $D_l$ the diagonal automorphisms are given by conjugation by the corresponding diagonal group defined in Section~\ref{diagonalgroup}. 

Let $K$ be a finite field of odd characteristic and $G=\mathcal L(K)$ be an adjoint Chevalley group over $K$ as defined in Section~\ref{chevalleygroups}.
Steinberg described the automorphisms of these groups.
 We have the following theorem~\cite[Theorem 12.5.1]{ca} and ~\cite{st1},
\begin{theorem}[Steinberg]
Let $G=\mathcal L(K)$ where $\mathcal L$ is simple and $K(=\mathbb
F_q)$ is a finite field. 
Let $\phi\in \Aut(G)$. Then there exist inner, diagonal, graph and
field automorphisms, denoted by
$\iota ,\delta, \gamma$ and $\theta$ respectively, such that $\phi = \iota \delta \gamma \theta$.
\end{theorem}
The automorphism groups of Chevalley groups over certain rings have been studied by Bunina~\cite{bu1, bu2}.

\section{Solving the word problem in $G$}\label{wordproblem}

We work with a finite field $K=\mathbb F_q$ of odd characteristic. Let $G$ be one of the following groups: $\text{SL}(l+1,q)$, $\text{O}(2l+1,q)$, $\text{Sp}(2l,q)$ or $\text{O}(2l,q)$ for $l\geq 2$. Following the notation from the theory of Chevalley groups we also call them $A_l, B_l, C_l$ or $D_l$ type respectively. 
We know that the group $G$ is generated by Chevalley generators listed in the Section~\ref{chgenerators}. In
fact, there are finite presentations for these groups due to Steinberg.
In computational group theory, one is always looking for algorithms
that solve the word problem. Algorithms for word problem are useful in other
programs in computational group theory, such as, the group recognition program and studying the
membership problems in finite groups. Extensive work on these programs
are being done by several people, most notably of those are Leedham-Green and O'Brien~\cite{lo} and Guralnick et.~al.~\cite{gkkl1, gkkl2, gkkl3}.  We need an (efficient) algorithm to write an element $g\in G$ as a product of generators, i.e., a solution to the word problem for an efficient implementation of the MOR cryptosystem.
 
In the case of groups of $A_l$ type, i.e., when $G$ is a special linear group, one has the well-known algorithm, the row-column
operations. One observes that the effect of multiplying by a Chevalley
generator on a matrix from left or right is either a row or a column operation
respectively. Using this algorithm one can start with any matrix $g\in\text{SL}(l+1,q)$ and get the identity matrix thus writing $g$ as a product of generators. One of the objective in this paper is to develop a similar algorithm for the groups of type $B_l, C_l$ and $D_l$ type.  

In general, one has the Bruhat decomposition for Chevalley groups which can be used to write any element in a normal form. 
Every element $g\in \mathcal L(K)$ has a unique expression ~\cite[Corollary
8.4.4]{ca} $u_1hn_{w}u$ where $u_1\in U, h\in H, w\in W$
and $u\in U_w^-$. Here we fixed a coset representative for each $w\in
W$ and denote it by $n_w$. The element $n:=hn_w$ belongs to $N$. 

Thus, the main objective of this section is to give an algorithm, in a similar
line as the row-column operations for $A_l$, to solve the word problem for other Chevalley groups. 

Cohen, Murray and Taylor~\cite{CMT} proposed a generalized algorithm
using the row-column operations, using a representation of Chevalley
groups. The key idea there was to bring down an element to a maximal
parabolic subgroup and repeat the process inductively. Here we use the natural
matrix representation of these groups. Thus our algorithm is more
direct and works with matrices explicitly and effectively. A novelty of our algorithm is that we do not need to assume that the Chevalley generators generate the group under consideration. Thus our algorithm proves independently the fact that these groups are generated by those generators.  

\subsection{ An algorithm for row-column operations for the groups of Lie type $C_l$ and $D_l$}
First we will deal with groups of $C_l$ and $D_l$ type. That is, we work with
groups $\text{Sp}(2l,q)$ and $\text{O}(2l,q)$. The Chevalley generators are described in Section~\ref{chgenerators}.
In general, we have three kind of Chevalley generators. For $1\leq i,j\leq l$
\begin{enumerate}
\item[CG1:] $\begin{pmatrix}R&\\ & \tra R^{-1}\end{pmatrix}$ where $R=I+te_{i,j}$; $i\neq j $.
\item[CG2:] $\begin{pmatrix} I& R \\ &I\end{pmatrix}$ where $R$ is either $t(e_{i,-j}+e_{j,-i})$ or $te_{i,-i}$ in the case of $C_l$ and $R$ is $t(e_{i,-j}-e_{j,-i})$ in the case of $D_l$.
\item[CG3:] $\begin{pmatrix} I&  \\ R &I\end{pmatrix}$ where $R$ is either $t(e_{-i,j}+e_{-j,i})$ or $te_{-i,i}$ in the case of $C_l$ and $R$ is $t(e_{-i,j}-e_{-j,i})$ in the case of $D_l$.
\end{enumerate}

Let $g=\begin{pmatrix}A&B\\C&D\end{pmatrix}$ be a $2l\times 2l$ matrix. Let us note the effect of multiplying $g$ by elements from above.
\begin{eqnarray*}
CG1:&\begin{pmatrix}R&\\ & \tra R^{-1}\end{pmatrix}\begin{pmatrix}A&B\\C&D\end{pmatrix}&=\begin{pmatrix}RA&RB \\ \tra R^{-1}C & \tra R^{-1} D\end{pmatrix}\\
&\begin{pmatrix}A&B\\C&D\end{pmatrix}\begin{pmatrix}R&\\ & \tra R^{-1}\end{pmatrix}&=\begin{pmatrix} AR&B\tra R^{-1} \\ CR & D\tra R^{-1}\end{pmatrix}.\\
CG2: & \begin{pmatrix}I&R\\ &I\end{pmatrix}\begin{pmatrix}A&B\\C&D\end{pmatrix}&=\begin{pmatrix}A+RC&B+RD \\ C & D\end{pmatrix} \\
&\begin{pmatrix}A&B\\C&D\end{pmatrix}\begin{pmatrix}I&R\\ &I \end{pmatrix}&=\begin{pmatrix} A&AR+B \\ C & CR+ D\end{pmatrix}.\\
CG3: &\begin{pmatrix}I&\\R &I\end{pmatrix}\begin{pmatrix}A&B\\C&D\end{pmatrix}&=\begin{pmatrix}A&B \\ RA+ C & RB+D\end{pmatrix}\\ 
&\begin{pmatrix}A&B\\C&D\end{pmatrix}\begin{pmatrix}I&\\ R&I \end{pmatrix}&=\begin{pmatrix} A+BR&B \\ C+DR & D\end{pmatrix}.
\end{eqnarray*}

\subsubsection{Algorithm}We produce a brief overview of the row-column operations for groups of type $\text{Sp}(2l,q)$ and $\text{O}(2l,q)$.
\begin{itemize}
\item[Step 1:]
\textbf{Input}: A matrix $g=\begin{pmatrix}A&B\\C&D\end{pmatrix}$ which belongs to $\text{Sp}(2l,q)$ or $\text{O}(2l,q)$. 

\noindent\textbf{Output}: The matrix $g_1=\begin{pmatrix}A_1&B_1\\C_1&D_1\end{pmatrix}$ is one of the following kind: 
\begin{enumerate} 
\item[a:] The matrix $C_1$ is a diagonal matrix $\diag(1,1,\ldots, 1, \lambda)$ and $A_1$ is $\begin{pmatrix} A_{11}&A_{12}\\ A_{21}&a_{22}\end{pmatrix}$ where $A_{11}$ is symmetric in the $\text{Sp}(2l,q)$ case and skew-symmetric in the $\text{O}(2l,q)$ case of size $l-1$. Furthermore, $A_{12} = \lambda \tra A_{21}$ in the $\text{Sp}(2l,q)$ case and $A_{12} = -\lambda \tra A_{21}$ in the $\text{O}(2l,q)$ case. 
\item[b:] The matrix $C_1$ is a diagonal matrix $\diag(1,1,\ldots,1,0,\ldots,0)$ with number of $1$s equal to $m$ and $A_1$ looks like $\begin{pmatrix}A_{11}&0\\ A_{21}& A_{22}\end{pmatrix}$ where $A_{11}$ is an $m\times m$ symmetric in the $\text{Sp}(2l,q)$ case and skew-symmetric in the $\text{O}(2l,q)$ case.
\item[c:] The matrices $B_1$ and $D_1$ are $l\times l$.
\end{enumerate}
 
\paragraph{\textbf{Justification}}: Observe that the effect of CG1 on $C$ is the
usual row-column operations. Thus we can reduce $C$ to the diagonal form and Corollary~\ref{corA} makes sure that $A$ has required form.

\item[Step 2:]
\textbf{Input:} matrix $g_1=\begin{pmatrix}A_1&B_1\\C_1&D_1\end{pmatrix}$.

\noindent\textbf{Output:} matrix $g_2=\begin{pmatrix}A_2&B_2\\0&\tra A_2^{-1}\end{pmatrix}$; $A_2$ is a diagonal matrix $\diag(1,1,\ldots,1,\lambda)$.

\noindent\textbf{Justification}: Observe the effect of CG2. It changes $A_1$ by $A_1+RC_1$. 
Using Lemma~\ref{lemmaC} we can make the matrix $A_1$ the zero matrix in the first case and $A_{11}$ the zero matrix in the second case. After that we make use of Lemma~\ref{lemmaD} to interchange the rows so that we get zero matrix at the place of $C_1$. If required use CG1 to make $A_1$ a diagonal matrix. The Lemma~\ref{lemmaB} ensures that $D_1$ becomes $\tra A^{-1}_2$. 

\item[Step 3:] 
\textbf{Input:}  matrix $g_2=\begin{pmatrix}A_2&B_2\\0&\tra A_2^{-1}\end{pmatrix}$; $A_2$ is a diagonal matrix $\diag(1,\ldots,1,\lambda)$.

\noindent\textbf{Output:} Matrix $g_3=\begin{pmatrix}A_2&0\\0&\tra A_2^{-1}\end{pmatrix}$; $A_2$ is diagonal matrix $\diag(1,1,\ldots,1,\lambda)$.

\noindent\textbf{Justification}: Using Corollary~\ref{corB} we see that the
matrix $B_2$ has certain form. We can use CG2 to make the matrix
$B_2$ a zero matrix because of Lemma~\ref{lemmaC}.

\item[Step 4:] 
\textbf{Input:} matrix $g_3=\diag(1,\ldots,1,\lambda,1,\ldots,1,\lambda^{-1})$.

\noindent\textbf{Output:} Identity matrix

\noindent\textbf{Justification}: In the case of $\text{Sp}(2l,q)$ the diagonal matrix can be written as a product of generators by first part of Lemma~\ref{lemmaE}. In the case of $\text{O}(2l,q)$, using second part of Lemma~\ref{lemmaE} we can reduce to $\diag(1,\ldots,1,\zeta,1,\ldots,1,\zeta^{-1})$  where $\zeta$ is a fixed non-square in $\mathbb F_q$. Thus multiplying with $d(\zeta)^{-1}$ we get the result.
\end{itemize}
\subsection{Time-complexity of the above algorithm} We establish that the time-complexity of the above algorithm is $\mathcal{O}(l^3)$.
\begin{itemize}
\item[] In Step 1, we are making $C$ a diagonal matrix by row-column operations. That has complexity $\mathcal{O}(l^3)$.
\item[] In Step 2, $A_1+RC_1$ is two field multiplications and two additions. In the worst case, it has to be done $l^2$ times and so the complexity is $\mathcal{O}(l^2)$.
\item[] Step 3 is similar to Step 2 above and has complexity $\mathcal{O}(l^2)$.
\item[] Step 4 has only a few steps that is independent of $l$.
 \end{itemize}
Then clearly, the time-complexity of our algorithm is $\mathcal{O}(l^3)$.
\subsection{Flowchart of the above algorithm}
The input to the algorithm is a $2l\times 2l$ matrix in $\text{Sp}(2l,q)$ or $\text{O}(2l,q)$ represented as blocks of size $l$.
\vskip5mm
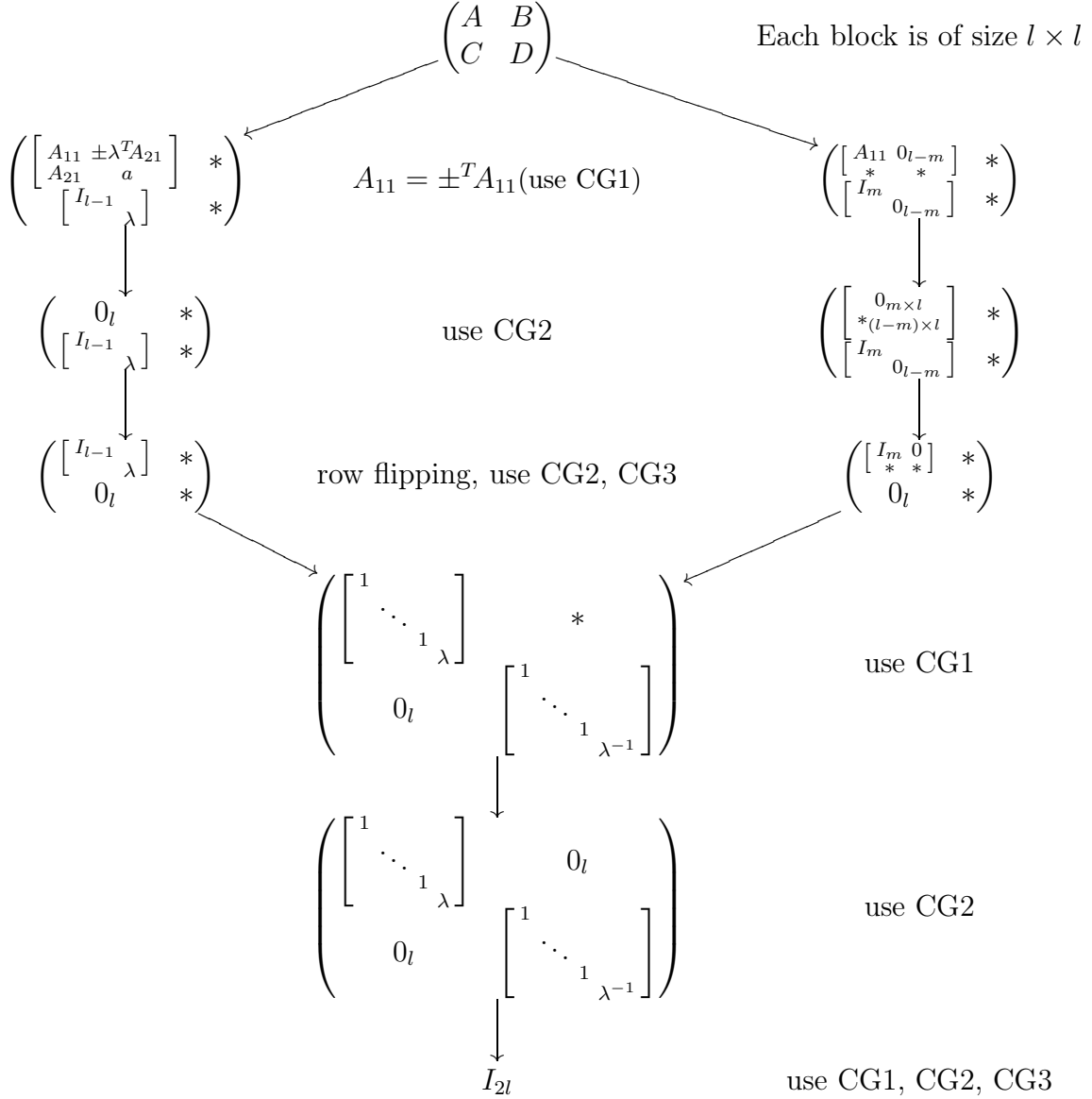
\begin{displaymath}
\xymatrix{
&*{\begin{pmatrix} A&B\\C&D\end{pmatrix}}\ar[ld]\ar[rd]& \text{Each
  block is of size}\; l\times l\\
*{\begin{pmatrix}{\left[\begin{smallmatrix} A_{11}&\pm \lambda \tr A_{21}\\A_{21}&a\end{smallmatrix}\right]}&*\\{\left[\begin{smallmatrix}I_{l-1}&\\&\lambda\end{smallmatrix}\right]}&*\end{pmatrix}}\ar[d]   & {\small A_{11}=\pm \tra A_{11}\text{(use CG1)} }  &   
*{\begin{pmatrix}{\left[\begin{smallmatrix} A_{11}& 0_{l-m}\\ *&*\end{smallmatrix}\right]}&*\\{\left[\begin{smallmatrix}I_{m}&\\&0_{l-m}\end{smallmatrix}\right]}&*\end{pmatrix}}\ar[d]\\
*{\begin{pmatrix} 0_l&*\\{\left[\begin{smallmatrix}I_{l-1}&\\&\lambda\end{smallmatrix}\right]}&*\end{pmatrix}}\ar[d]  & \text{use CG2} & *{\begin{pmatrix}{\left[\begin{smallmatrix} 0_{m\times l} \\ *_{(l-m)\times l}\end{smallmatrix}\right]}&*\\{\left[\begin{smallmatrix}I_{m}&\\&0_{l-m}\end{smallmatrix}\right]}&*\end{pmatrix}}\ar[d]\\ 
*{\begin{pmatrix}
    {\left[\begin{smallmatrix}I_{l-1}&\\&\lambda\end{smallmatrix}\right]}&*\\
    0_l & *\end{pmatrix}}\ar[rd]  & \text{row flipping, use CG2, CG3} & *{\begin{pmatrix}{\left[\begin{smallmatrix} I_m & 0 \\ *&*\end{smallmatrix}\right]}&*\\ 0_l&*\end{pmatrix}}\ar[ld]\\ 
&*{\begin{pmatrix}{\left[\begin{smallmatrix}1&&&\\ &\ddots &&\\&&1&\\&&&\lambda\end{smallmatrix}\right]} & * \\ 0_l & {\left[\begin{smallmatrix}1&&&\\ &\ddots &&\\&&1&\\&&&\lambda^{-1}\end{smallmatrix}\right]}\end{pmatrix}}\ar[d]& \text{use CG1} \\ 
&*{\begin{pmatrix}{\left[\begin{smallmatrix}1&&&\\ &\ddots &&\\&&1&\\&&&\lambda\end{smallmatrix}\right]} & 0_l \\ 0_l & {\left[\begin{smallmatrix}1&&&\\ &\ddots &&\\&&1&\\&&&\lambda^{-1}\end{smallmatrix}\right]}\end{pmatrix}}\ar[d]& \text{use CG2}\\
&I_{2l}& \text{use CG1, CG2, CG3}
}
\end{displaymath}

\subsection{Useful lemmas}

In this section we set notation and prove lemmas which were used (and will be used) to justify the above algorithm (and the later algorithm). Some of these might be well known to experts but we include them here for the convenience of the reader. We make use of the following while computing with matrices:
\[e_{i,j}e_{k,l}=\delta_{jk}e_{i,l}\;\text{where}\; \delta_{jk}\;\text{is the Kronecker delta}.\]

\begin{lemma}\label{lemmaA}
Let $Y=\diag(1,\ldots,1,\lambda,\ldots,\lambda)$ of size $l$ with number of $1$s equal to $m<l$. Let $X$ be a matrix such that $YX$ is symmetric (skew-symmetric) then $X$ is of the form $\begin{pmatrix}X_{11}& \lambda \tr X_{21}\\ X_{21}&X_{22}\end{pmatrix}$ where $X_{11}$ is symmetric (skew symmetric) and $X_{12}=\lambda \tr X_{21}$ ($X_{12}=- \lambda \tr X_{21}$).
\end{lemma}
\begin{proof}
We observe that the matrix $YX=\begin{pmatrix}X_{11}& X_{12}\\\lambda X_{21}&\lambda X_{22}\end{pmatrix}$. The condition that $YX$ is symmetric implies $X_{11}$ (and $X_{22}$ if $\lambda\neq 0$) is symmetric and $X_{12}=\lambda \tra X_{21}$. 
\end{proof}
\begin{corollary}\label{corA}
Let $g=\begin{pmatrix}A&B\\C&D\end{pmatrix}$ be either in $\text{Sp}(2l,q)$ or $\text{O}(2l,q)$.
\begin{enumerate}
\item  If $C$ is a diagonal matrix $\diag(1,1,\ldots,1,0,\ldots,0)$ with number of $1$s equal to $m<l$ then the matrix $A$ has to be of the form $\begin{pmatrix}A_{11}&0\\ A_{21}& A_{22}\end{pmatrix}$ where $A_{11}$ is an $m\times m$ symmetric if $g$ is symplectic  and is skew-symmetric if $g$ is orthogonal.
\item If $C$ is a diagonal matrix $\diag(1,1,\ldots,1,\lambda)$ then the matrix $A$ has to be of the form $\begin{pmatrix}A_{11}&\lambda \tr A_{21}\\ A_{21}& A_{22}\end{pmatrix}$ where $A_{11}$ is an $(l-1)\times (l-1)$ symmetric if $g$ is symplectic  and is skew-symmetric if $g$ is orthogonal.
\end{enumerate}
\end{corollary}
\begin{proof}
We use the condition that $g$ satisfies $\tra g\beta g=\beta$.
\begin{eqnarray*}
\tra g \beta g &=& \begin{pmatrix}\tr A&\tra C\\\tra B&\tra D\end{pmatrix} \begin{pmatrix}&I\\\pm I&\end{pmatrix} \begin{pmatrix}A&B\\C&D\end{pmatrix}\\
&=& \begin{pmatrix}\pm \tra C&\tr A\\\pm \tra D&\tra B\end{pmatrix}\begin{pmatrix}A&B\\C&D\end{pmatrix} =\begin{pmatrix}\pm \tra CA+\tr AC&*\\ *&*\end{pmatrix}
\end{eqnarray*}
This gives $\pm \tra CA+\tr AC=0$ which means $CA$ is symmetric (note $C=\tra C$ as $C$ is diagonal) if $g$ is symplectic and is skew-symmetric if $g$ is orthogonal. The Lemma~\ref{lemmaA} gives the required form for $A$.
\end{proof}
\begin{corollary}\label{corB}
Let $g=\begin{pmatrix}A&B\\0&A^{-1}\end{pmatrix}$ where $A=\diag(1,\ldots,1,\lambda)$ be an element of either $\text{Sp}(2l,q)$ or $\text{O}(2l,q)$ then the matrix $B$ is of the form  $\begin{pmatrix}B_{11}&\lambda \tr B_{21}\\ B_{21}& B_{22}\end{pmatrix}$ where $B_{11}$ is a symmetric matrix of size $l-1$ if $g$ is symplectic  and is skew-symmetric if $g$ is orthogonal.
\end{corollary}
\begin{proof}
Yet again, we use the condition that $g$ satisfies $\tra g\beta g=\beta$ and $A=\tr A$.
\begin{eqnarray*}
\tra g \beta g &=& \begin{pmatrix} A& \\\tra B& A^{-1}\end{pmatrix} \begin{pmatrix}&I\\\pm I&\end{pmatrix} \begin{pmatrix}A&B\\&A^{-1}\end{pmatrix}\\
&=& \begin{pmatrix}& A\\\pm A^{-1}&\tra B\end{pmatrix}\begin{pmatrix}A&B\\&A^{-1}\end{pmatrix} =\begin{pmatrix}& I\\ \pm I& \pm A^{-1}B+\tra BA^{-1} \end{pmatrix}
\end{eqnarray*}
This gives $\pm A^{-1}B+\tra BA^{-1}=0$ which means $A^{-1}B$ is symmetric if $g$ is symplectic and is skew-symmetric if $g$ is orthogonal. Then Lemma~\ref{lemmaA} gives the required form for $B$.
\end{proof}
\begin{lemma}\label{lemmaB}
Let $g=\begin{pmatrix}A&*\\0&D\end{pmatrix}\in \text{GL}(2l,q)$. If $g$ belongs to $\text{Sp}(2l,q)$ or $\text{O}(2l,q)$ then $D=\tra A^{-1}$. 
\end{lemma}
\begin{proof}
We use $\tra g \beta g=\beta$.
 \begin{eqnarray*}
\begin{pmatrix}&I\\\pm I&\end{pmatrix} = \beta=\tra g \beta g &=& \begin{pmatrix}\tr A& 0 \\ *&\tra D\end{pmatrix} \begin{pmatrix}&I\\\pm I&\end{pmatrix} \begin{pmatrix}A& *\\ 0 &D\end{pmatrix}\\
&=& \begin{pmatrix} 0&\tr A\\\pm \tra D& * \end{pmatrix}\begin{pmatrix}A& *\\0&D\end{pmatrix} =\begin{pmatrix} 0 & \tr AD\\ \pm \tra DA&*\end{pmatrix}
\end{eqnarray*}
This gives $\tr AD=I$.
\end{proof}

\begin{lemma}\label{lemmaC}
Let $Y=\diag(1,1,\ldots,1,\lambda)$ be of size $l$ where $\lambda\neq 0$ and $X=(x_{ij})$ be a matrix such that $YX$ is symmetric (skew-symmetric). Then $X=(R_1+R_2+\ldots)Y$ where each $R_m$ is of the form $t(e_{i,j}+e_{j,i})$ for some $i < j$ or of the form $te_{i,i}$ for some $i$ (in the case of skew-symmetric each $R_m$ is of the form $t(e_{i,j}-e_{j,i})$ for some $i < j$). 
\end{lemma}
\begin{proof}
Since $YX$ is symmetric, the matrix $X$ is of the following form (see Lemma~\ref{lemmaA}): 
$\begin{pmatrix}X_{11}&X_{12}\\ X_{21} & x_{nn}\end{pmatrix}$ where $X_{11}$ is symmetric and  $X_{21}$ is a row of size $l-1$ $(x_{l1} x_{l2}\cdots x_{l,l-1})$ and $X_{12} = \lambda \tr X_{21}$. Clearly any such matrix is sum of the matrices of the form $RY$. A similar calculation proves the result in the skew-symmetric case.   
\end{proof}

We need certain Weyl group elements which can be used for switching rows.
\begin{lemma}\label{lemmaD}
With the indexing of basis as $1,\ldots,l,-1,\ldots,-l$,
for any matrix $g$ in $\text{Sp}(2l,q)$ or $\text{O}(2l,q)$, the $i\textsuperscript{th}$ row can be interchanged with $-i\textsuperscript{th}$ row with possibly a sign change. Further, we can do the same in $\text{O}(2l+1,q)$.
\end{lemma}
\begin{proof}
For the symplectic group $\text{Sp}(2l,q)$ consider the following root generators: $x_{i,-i}=I+e_{i,-i}$ and $y_{i,-i}=I-e_{-i,i}$. 
Then the element $w_{i,-i}=x_{i,-i}y_{i,-i}x_{i,-i}$ is in the Weyl group and multiplication by this element to a matrix $g$ has desired property.
\begin{eqnarray*}
w_{i,-i}&=& x_{i,-i}y_{i,-i}x_{i,-i} = (I+e_{i,-i})(I-e_{-i,i})(I+e_{i,-i})\\
&=& (I+e_{i,-i}-e_{-i,i}-e_{i,i})(I+e_{i,-i}) \\
&=& I +e_{i,-i}- e_{-i,i}-e_{i,i} - e_{-i,-i}.
\end{eqnarray*}
In the matrix form:
$$\begin{pmatrix} 1& & 1&\\ &1&&\\&&1&\\&&&1\end{pmatrix}
\begin{pmatrix}1&&&\\ &1&&\\ -1&&1&\\&&&1 \end{pmatrix}
\begin{pmatrix} 1& & 1&\\ &1&&\\&&1&\\&&&1\end{pmatrix}
= \begin{pmatrix} & & 1&\\ &1&&\\-1&&&\\&&&1\end{pmatrix}.
$$

For the orthogonal group $\text{O}(2l,q)$ consider the following root generators: $x_{ij}=I+(e_{i,-j}-e_{j,-i})$ and $y_{ij}=I+(e_{-i,j}-e_{-j,i})$ for $i<j$. Then the element $w_{ij}=x_{ij}y_{ij}x_{ij}$ is in the Weyl group and multiplication by this element to a matrix $g$ changes $i\textsuperscript{th}$ row with $-j\textsuperscript{th}$ row with a sign change and $j\textsuperscript{th}$ row with $-i\textsuperscript{th}$ row with a sign change simultaneously.
\begin{eqnarray*}
w_{ij}&=& x_{ij}y_{ij}x_{ij}=(I+e_{i,-j}-e_{j,-i})(I+e_{-i,j}-e_{-j,i})(I+e_{i,-j}-e_{j,-i})\\
&=& (I+e_{-i,j}-e_{-j,i}+e_{i,-j}+e_{i,-j}e_{-i,j}-e_{i,-j}e_{-j,i}-e_{j,-i}-e_{j,-i}e_{-i,j}\\&&+e_{j,-i}e_{-j,i})(I+e_{i,-j}-e_{j,-i})\\
&=& (I+e_{-i,j}-e_{-j,i}+e_{i,-j}-e_{i,i}-e_{j,-i}-e_{j,j})(I+e_{i,-j}-e_{j,-i})\\
&=& I-e_{i,i}-e_{j,j}-e_{-i,-i}-e_{-j,-j}+e_{i,-j}-e_{j,-i}+e_{-i,j}-e_{-j,i}.
\end{eqnarray*}
In the matrix form:
$$\begin{pmatrix}1&&&1\\&1&-1& \\&&1&\\&&&1 \end{pmatrix}\begin{pmatrix}1&&&\\&1&& \\&1&1&\\-1&&&1 \end{pmatrix}\begin{pmatrix}1&&&1\\&1&-1& \\&&1&\\&&&1 \end{pmatrix}=\begin{pmatrix}&&&1\\&&-1& \\&1&&\\-1&&& \end{pmatrix}.$$
Also since $\text{GL}(l,q)$ embeds inside $\text{O}(2l,q)$ via $A\mapsto \begin{pmatrix}A&&\\&\tr A^{-1}\end{pmatrix}$ the CG1 generators generate the subgroup $\text{SL}(l,q)$ and we have corresponding Weyl group elements, $\sigma_{ij}=I-e_{i,i}-e_{j,j}+e_{i,j}-e_{j,i}-e_{-i,-i}-e_{-j,-j}+e_{-i,-j}-e_{-j,-i}$ which interchanges $i\textsuperscript{th}$ to $j\textsuperscript{th}$ row and $-i\textsuperscript{th}$ to $-j\textsuperscript{th}$ row simultaneously with a sign change. 
We have the extra generator $w_l\in O(2l,q)$ which interchanges $l\textsuperscript{th}$ row with $-l\textsuperscript{th}$ row with a sign change. We can compute and check that
$w_{l-1}=w_l\sigma_{l,l-1}w_{l,l-1}= I-e_{l-1,l-1}-e_{-(l-1),-(l-1)}-e_{(l-1),-(l-1)}-e_{-(l-1),(l-1)}$ which interchanges $l-1\textsuperscript{th}$ row
with $-(l-1)\textsuperscript{th}$ row (possibly with a sign change) and inductively we can produce $w_i$ which interchanges $i\textsuperscript{th}$ row with $-i\textsuperscript{th}$ row possibly with a sign change. In the matrix form:
\begin{eqnarray*}
w_3\sigma_{23}w_{23}&=&\begin{pmatrix}1&&&&&\\&1&&&& \\&&&&&-1\\&&&1&&\\&&&&1&\\&&-1&&& \end{pmatrix}\begin{pmatrix}1&&&&&\\&&-1&&& \\&1&&&&\\&&&1&&\\ &&&&&-1\\&&&&1& \end{pmatrix}\begin{pmatrix}1&&&&&\\&&&&&-1 \\&&&&1&\\&&&1&&\\&&-1&&&\\&1&&&& \end{pmatrix}\\
&=&\begin{pmatrix}1&&&&&\\&&&&-1& \\&&1&&&\\&&&1&&\\ &-1&&&&\\&&&&&1 \end{pmatrix}=w_2.
\end{eqnarray*}
Further notice that $\text{O}(2l,q)$ is embedded inside $\text{O}(2l+1,q)$. Thus we can do the same in $\text{O}(2l+1,q)$ as well.
\end{proof}

\begin{lemma}\label{lemmaE}
\begin{enumerate}
\item In the case of $\text{Sp}(2l,q)$, the element $\diag(\underbrace{1,\ldots,1\lambda}_l,\underbrace{1,\ldots,1,\lambda^{-1}}_l)$ is a product of Chevalley generators.
\item In the case of $\text{O}(2l,q)$, the element $\diag(\underbrace{1,\ldots,1\lambda}_l,\underbrace{1,\ldots,1,\lambda^{-1}}_l)$ is a product of Chevalley generators where $\lambda\in {\mathbb F_q^\times}^2$.
\item In the case of $\text{O}(2l+1,q)$ diagonal elements $\diag(1,\underbrace{1,\ldots,1\lambda}_l,\underbrace{1,\ldots,1,\lambda^{-1}}_l)$ where $\lambda\in\mathbb{F}_q^{\times 2}$ and $\diag(-1,1,\ldots,1)$ are a product of Chevalley generators.
\end{enumerate}
\end{lemma}
\begin{proof}
In the case of $\text{Sp}(2l,q)$, we compute $w_{l,-l}(t)=(I+te_{l,-l})(I-t^{-1}e_{-l,l})(I+te_{l,-l}) = I-e_{l,l}-e_{-l,-l}+te_{l,-l}-t^{-1}e_{-l,l}$ and then compute $h_l(\lambda)=w_{l,-l}(\lambda)w_{l,-l}(-1)$ which is the required element.

In the case of $\text{O}(2l,q)$, we compute $w_{l-1,-l}(t)=(I+te_{l-1,-l}-te_{l,-(l-1)})(I+t^{-1}e_{-(l-1),l}-t^{-1}e_{-l,l-1})(I+te_{l-1,-l}-te_{l,-(l-1)})=I+t^{-1}e_{-(l-1),l}-e_{-(l-1),-(l-1)}-t^{-1}e_{-l,l-1}-e_{-l,-l}+te_{l-1,-l}-e_{l-1,l-1}-te_{l,-(l-1)}-e_{l,l}$ and \[h_{l-1,-l}(t)=w_{l-1,l}(t)w_{l-1,l}(-1)=\diag(\underbrace{1,\ldots,1,t,t}_l,\underbrace{1\ldots,1,t^{-1},t^{-1}}_l)\]. Similarly we compute
$\sigma_{l-1,l}(t)=(I+te_{l-1,l}+te_{-l,-(l-1)})(I-t^{-1}e_{l,l-1}-t^{-1}e_{-(l-1),-l})(I+te_{l-1,l}+te_{-l,-(l-1)})$ and $h_{l-1,l}(t)=\sigma_{l-1,l}(t)\sigma_{l-1,l}(-1)=\diag(\underbrace{1,\ldots,1,t,t^{-1}}_l,\underbrace{1,\ldots,1,t^{-1},t}_l)$.

In the matrix form:
\begin{eqnarray*}
w_{2,-3}(t) &=&\begin{pmatrix}1&&&&&\\ &1&&&&t\\&&1&&-t&\\&&&1&&\\&&&&1&\\&&&&&1\end{pmatrix}\begin{pmatrix}1&&&&&\\ &1&&&&\\&&1&&&\\&&&1&&\\&&t^{-1}&&1&\\&-t^{-1}&&&&1\end{pmatrix}\begin{pmatrix}1&&&&&\\ &1&&&&t\\&&1&&-t&\\&&&1&&\\&&&&1&\\&&&&&1\end{pmatrix}\\
&=& \begin{pmatrix}1&&&&&\\ &0&&&&t\\&&0&&-t&\\&&&1&&\\&&t^{-1}&&0&\\&-t^{-1}&&&&0\end{pmatrix}.
\end{eqnarray*}
and 
\begin{eqnarray*}
h_{2,-3}(t)=w_{2,-3}(t)w_{2,-3}(-1)&=& \begin{pmatrix}1&&&&&\\ &0&&&&t\\&&0&&-t&\\&&&1&&\\&&t^{-1}&&0&\\&-t^{-1}&&&&0\end{pmatrix}\begin{pmatrix}1&&&&&\\ &0&&&&-1\\&&0&&1&\\&&&1&&\\&&-1&&0&\\&1&&&&0\end{pmatrix}\\
&=&\begin{pmatrix}1&&&&&\\ &t&&&&\\&&t&&&\\&&&1&&\\&&&&t^{-1}&\\&&&&&t^{-1}\end{pmatrix}.
\end{eqnarray*}
Furthermore, 

$\sigma_{23}(t)=\begin{pmatrix}1&&&&&\\&0&t&&&\\&-t^{-1}&0&&&\\ &&&1&&\\&&&&0&-t^{-1}\\&&&&t&0\end{pmatrix}$
and 
$h_{23}(t)=\sigma_{23}(t)\sigma_{23}(-1)= \begin{pmatrix}1&&&&&\\&t&&&&\\&&t^{-1}&&&\\ &&&1&&\\&&&&t^{-1}&\\&&&&&t\end{pmatrix} $.
Thus multiplying $h_{l-1,-l}(t)$ and $h_{l-1,l}(t^{-1})$ we get the required result.

In the case of $\text{O}(2l+1,q)$ we compute $w_{l,0}=x_{l,0}(1)x_{0,l}(-1)x_{l,0}(1) = I-e_{-l,-l}-e_{-l,l}-e_{l,l}-2e_{0,0}-e_{l,-l}$ and multiply it with $w_l$ to get the required matrix.
\end{proof}

\begin{lemma}\label{lemmaF}
Let $g=\begin{pmatrix}\alpha&X&*\\ *&A&*\\ *&C&*\end{pmatrix}$ be in $\text{O}(2l+1,q)$.
\begin{enumerate}
\item If $C=\diag(1,\ldots,1,\lambda)$ and $X=0$ then $A$ is of the form $\begin{pmatrix}A_{11}&-\lambda \tra A_{21}\\ A_{21}&a\end{pmatrix}$ with $A_{11}$ skew-symmetric.
\item  If $C=\diag(1,\ldots,1,0,\ldots,0)$ with number of $1$s equal $m<l$ and $X$ has first $m$ entries $0$ then $A$ is of the form $\begin{pmatrix}A_{11}& 0 \\ *&*\end{pmatrix}$ with $A_{11}$ skew-symmetric. 
\end{enumerate}
\end{lemma}
\begin{proof} We use the equation $\tra g \beta g=\beta$ and get $2\tra XX= -(CA+\tra AC)$. In the first case $X=0$, so we can use~\ref{corA} to get required form for $A$. In the second case we note that $\tra XX$ has top-left block $0$ and get the required form.
\end{proof}

\begin{lemma}\label{lemmaI}
Let $g=\begin{pmatrix}\alpha&X&Y\\ *&A&*\\ *&0&D\end{pmatrix}$ be in $\text{O}(2l+1,q)$ then $X=0$ and $D=\tra A^{-1}$.
\end{lemma}
\begin{proof} We compute $\tra g\beta g=\beta$ and get $2\tra XX=0$ and $2\tra XY+\tra AD=I$. This gives the required result.
\end{proof}

\begin{lemma}\label{lemmaG}
Let $g=\begin{pmatrix}\alpha& 0&Y\\ 0&A&B\\F&0&D\end{pmatrix}$, with $A$ an invertible diagonal matrix, be in $\text{O}(2l+1,q)$ then $\alpha^2=1, F=0=Y$, $D=A^{-1}$ and $\tra DB+\tra BD=0$. 
\end{lemma}
\begin{proof}
\begin{eqnarray*}
\tra g\beta g &=& \begin{pmatrix}\alpha& 0&\tra F\\ 0&\tra A& 0 \\\tra Y& \tra B&\tra D\end{pmatrix} \begin{pmatrix}2& &\\ &&I\\&I&\end{pmatrix}\begin{pmatrix}\alpha& 0&Y\\ 0&A&B\\F&0&D\end{pmatrix}\\
&=& \begin{pmatrix}2\alpha^2& \tra F A &2\alpha Y+\tra FB\\ \tra AF &0 &\tra AD\\2\alpha \tra Y +\tra BF& \tra DA & 2\tra YY +\tra DB+\tra BD\end{pmatrix}.
\end{eqnarray*}
Equating this with $\beta$ we get the required result.
\end{proof}

\begin{lemma}\label{lemmaH}
Let $g=\begin{pmatrix}\pm 1& 0&0\\ 0&A&B\\0&0&A^{-1}\end{pmatrix}\in \text{O}(2l+1,q)$ where $A=\diag(1,\ldots,1,\lambda)$ is invertible then $B$ is of the form $\begin{pmatrix}B_{11}&\lambda^{-1}\tra B_{21}\\ B_{21} & b \end{pmatrix}$.
\end{lemma}
\begin{proof}
This follows from the computation in the Lemma~\ref{lemmaG} that $A^{-1}B+\tra BA^{-1}=0$ and Corollary~\ref{corA}.
\end{proof}
\subsection{ An algorithm for row-column operations for the groups of Lie type $B_l$}

Here we work with the group $\text{O}(2l+1,q)$. Recall that the basis will be indexed by $0, 1, \ldots, l, -1,\ldots, -l$. The Chevalley generators are described in the Section~\ref{chgenerators}. In general, we have four kind of Chevalley generators. For $1\leq i,j\leq l$
\begin{enumerate}
\item[CG1:] $\begin{pmatrix}1&&\\ &R&\\ && \tra R^{-1}\end{pmatrix}$ where $R=I+te_{i,j}$; $i\neq j$.
\item[CG2:] $\begin{pmatrix} 1&&\\ &I& R \\ &&I\end{pmatrix}$ where  $R$ is $t(e_{i,-j}-e_{j,-i})$; $i < j$.
\item[CG3:] $\begin{pmatrix} 1&&\\ &I&  \\ &R &I\end{pmatrix}$ where $R$ is $t(e_{-i,j}-e_{-j,i})$; $i < j$.
\item[CG4:]  $I+t(2e_{i0}-e_{0,-i})-t^2e_{i,-i}, I+t(-2e_{-i,0}+e_{0i})-t^2e_{-i,i}$.
\end{enumerate}

We observe that CG1, CG2 and CG3 generate the subgroup $\text{O}(2l,q)$ of $\text{O}(2l+1,q)$ given by $x\mapsto \begin{pmatrix} 1&\\&x\end{pmatrix}$.
Let $g=\begin{pmatrix}\alpha&X&Y\\ E& A&B\\F&C&D\end{pmatrix}$ be a $(2l+1)\times (2l+1)$ matrix where $A,B,C,D$ are $l\times l$ matrices. The matrices $X=(X_1,X_2,\ldots,X_l)$, $Y=(Y_1,Y_2,\ldots,Y_l)$, $E=\tr(E_1,E_2,\ldots,E_l)$ and $F=\tr(F_1,F_2,\ldots,F_l)$. Let $\alpha\in\mathbb{F}_q$. Let us note the effect of multiplication by elements of one of the types from above.
\begin{eqnarray*}
CG1:
&\begin{pmatrix}1&&\\ &R&\\ && \tra R^{-1}\end{pmatrix}\begin{pmatrix}\alpha & X&Y\\E& A&B\\F& C&D\end{pmatrix}&=\begin{pmatrix}\alpha &X&Y\\ RE&RA&RB \\ \tra R^{-1}F&\tra R^{-1}C & \tra R^{-1} D\end{pmatrix}\\
&\begin{pmatrix}\alpha & X&Y\\ E&A&B\\F&C&D\end{pmatrix}\begin{pmatrix}1&&\\ &R&\\ && \tra R^{-1}\end{pmatrix}&=\begin{pmatrix} \alpha& XR & Y\tra R^{-1}\\ E&AR&B\tra R^{-1} \\ F& CR & D\tra R^{-1}\end{pmatrix}.\\
\end{eqnarray*}
\begin{eqnarray*}
CG2: & \begin{pmatrix}1&&\\ &I&R\\ &&I\end{pmatrix}\begin{pmatrix}\alpha &X&Y\\ E&A&B\\F&C&D\end{pmatrix}&=\begin{pmatrix}\alpha &X&Y\\ E+RF& A+RC&B+RD \\ F& C & D\end{pmatrix} \\
&\begin{pmatrix}\alpha& X&Y\\ E& A&B\\F& C&D\end{pmatrix}\begin{pmatrix}1&&\\ &I&R\\ &&I \end{pmatrix}&=\begin{pmatrix} \alpha& X& XR+Y\\ E&A&AR+B \\ F& C & CR+ D\end{pmatrix}.\\
\end{eqnarray*}
\begin{eqnarray*}
CG3: &\begin{pmatrix}1&&\\ &I&\\&R &I\end{pmatrix}\begin{pmatrix}\alpha & X&Y\\ E&A&B\\F&C&D\end{pmatrix}&=\begin{pmatrix}\alpha&X&Y\\ E&A&B \\ RE+F&RA+ C & RB+D\end{pmatrix}\\ 
&\begin{pmatrix}\alpha& X&Y\\ E&A&B\\F&C&D\end{pmatrix}\begin{pmatrix}1&&\\ &I&\\ &R&I \end{pmatrix}&=\begin{pmatrix}\alpha& X+YR&Y\\E& A+BR&B \\ F&C+DR & D\end{pmatrix}.\\
\end{eqnarray*}
CG4: We only write equations that we need.
\begin{itemize}
\item Let the matrix $g$ has $C=\diag(d_1,\ldots,d_l)$.
$$[(I+te_{0,-i}-2te_{i,0}-t^2e_{i,-i})g]_{0,i}= X_i+td_i $$
$$ [g(I+te_{0,-i}-2te_{i,0}-t^2e_{i,-i})]_{-i,0}= F_i-2td_i.$$
\item Let the matrix $g$ has $A=\diag(d_1,\ldots,d_l)$.
$$[(I+te_{0,i}-2te_{-i,0}-t^2e_{-i,i})g]_{0,i}= X_i-td_i $$
$$[g(I+te_{0,-i}-2te_{i,0}-t^2e_{i,-i})]_{i,0}= E_i-2td_i.$$
\end{itemize}
\subsubsection{The Algorithm}An overview of the algorithm is as follows: 
\begin{itemize}
\item[Step 1:] 
\textbf{Input:} matrix $g=\begin{pmatrix}\alpha&X&Y\\ E&A&B\\F&C&D\end{pmatrix}$ which belongs to $\text{O}(2l+1,q)$; 

\noindent\textbf{Output:} matrix $g_1=\begin{pmatrix}\alpha & X_1 &Y_1 \\ E_1& A_1 &B_1\\ F_1& C_1&D_1\end{pmatrix}$ of one of the following kind:
\begin{enumerate} 
\item[a:] $C_1$ is a diagonal matrix $\diag(1,\ldots, 1, \lambda)$ with $\lambda\neq 0$. 
\item[b:] $C_1$ is a diagonal matrix $\diag(1,\ldots,1,0,\ldots,0)$ with number of $1$s equal to $m$ and $m<l$.
\end{enumerate} 

\noindent\textbf{Justification}: Using CG1 we can do row and column operations on $C$.
 
\item[Step 2:] 
\textbf{Input:} matrix $g_1=\begin{pmatrix}\alpha & X_1 & Y_1 \\ E_1& A_1 &B_1\\ F_1& C_1&D_1 \end{pmatrix}$. 

\noindent\textbf{Output:} matrix $g_2=\begin{pmatrix} \alpha_2 & X_2 & Y_2 \\ E_2& A_2 &B_2\\ F_2& C_2& D_2\end{pmatrix}$ of one of the following kind:    
\begin{enumerate} 
\item[a:] $C_2$ is $\diag(1,1,\ldots, 1, \lambda)$ with $\lambda\neq 0$, $X_2=0=F_2$ and $A_2$ is of the form $\begin{pmatrix} A_{11}&A_{12}\\ A_{21}&a_{22}\end{pmatrix}$ where $A_{11}$ is skew-symmetric of size $l-1$ and $A_{12} = -\lambda \tra A_{21}$. 
\item[b:] $C_2$ is $\diag(1,\ldots,1,0,\ldots,0)$ with number of $1$s equal to $m$; $X_2$ and $F_2$ have first $m$ entries $0$, and $A_2$ is of the form $\begin{pmatrix}A_{11}&0\\ A_{21}& A_{22}\end{pmatrix}$ where $A_{11}$ is an $m\times m$ skew-symmetric.
\end{enumerate}
 
\noindent\textbf{Justification}: Once we have $C_1$ in diagonal form we use CG4 to change $X_1$ and $F_1$ in the required form. Then Lemma~\ref{lemmaF} makes sure that $A_1$ has required form.

\item[Step 3:]
\textbf{Input:} matrix $g_2=\begin{pmatrix} \alpha_2 & X_2 & Y_2 \\ E_2& A_2 &B_2\\ F_2& C_2& D_2\end{pmatrix}$.

\noindent\textbf{Output:} 
\begin{enumerate} 
\item[a:] matrix $g_3=\begin{pmatrix}\alpha_3 & 0 & Y_3 \\ E_3& 0 &B_3\\ 0& C_3& D_3 \end{pmatrix}$ where $C_3$ is $\diag(1,1,\ldots, 1, \lambda)$. 
\item[b:] matrix $g_3=\begin{pmatrix}\alpha_3 & X_3 & Y_3 \\ E_3& A_3 &B_3\\ F_3& C_3& D_3 \end{pmatrix}$ where $C_3$ is $\diag(1,\ldots,1,0,\ldots,0)$ with number of $1$s equal to $m$; $X_3$ and $F_3$ have first $m$ entries $0$, and $A_3$ is of the form $\begin{pmatrix}0&0\\ A_{21}& A_{22}\end{pmatrix}$.
\end{enumerate}

\noindent\textbf{Justification}: Observe the effect of CG2 and the Lemma~\ref{lemmaC} ensures the required form.

\item[Step 4:] \textbf{Input:} $g_3=\begin{pmatrix}\alpha_3 & X_3 & Y_3 \\ E_3& A_3 &B_3\\ F_3& C_3& D_3 \end{pmatrix}$.

\noindent\textbf{Output:} $g_4=\begin{pmatrix}\pm 1&0&0\\ 0&A_4& B_4 \\0&0&A_4^{-1}\end{pmatrix}$  with $A_4$ diagonal matrix $\diag(1,\ldots,1,\lambda)$. 

\noindent\textbf{Justification}: In the first case, interchange rows $i$ and $-i$ for all $1\leq i\leq l$. Now the matrix is in the form so that we can apply Lemma~\ref{lemmaG} and get the required result. In the second case we interchange $i$ with $-i$ for $1\leq i \leq m$. This will make $C_3=0$. Then if needed we use CG1 on $A_3$ to make it diagonal. The Lemma~\ref{lemmaI} ensures that $A_3$ has full rank. Further we can use $CG4$ to make $X_3=0$ and $E_3=0$. The Lemma~\ref{lemmaG} gives the required form.

\item[Step 5:] \textbf{Input:} $g_4=\begin{pmatrix}\pm 1&0&0\\ 0&A_4&B_4\\0&0&A_4^{-1}\end{pmatrix}$  with $A_4=\diag(1,\ldots,1,\lambda)$.

\noindent\textbf{Output:} $g_5=\diag(\pm 1, 1\ldots, 1, \lambda, 1, \ldots, 1,\lambda^{-1})$. 

\noindent\textbf{Justification}: Lemma~\ref{lemmaH} ensures that $B_4$ is of a certain kind. We can use CG2 to make $B_4=0$.

\item[Step 6:] \textbf{Input:} matrix $\diag(\pm 1, 1,\ldots,1,\lambda,1,\ldots,1,\lambda^{-1})$.

\noindent\textbf{Output:} Identity matrix.

\noindent\textbf{Justification}: Write $\lambda$ as $\zeta$ times a square and use the third part of Lemma~\ref{lemmaE} to reduce the matrix to $\diag(1, 1,\ldots,1,\zeta,1,\ldots,1,\zeta^{-1})$  where $\zeta$ is a fixed non-square. Now multiplying with $d(\zeta)$ we get the required result. 
\end{itemize}

\section{Security of the proposed MOR cryptosystem}
The purpose of this section is to show that for a secure MOR cryptosystem over the classical Chevalley groups we have to look at automorphisms that act by conjugation, like the inner automorphisms. There are other automorphisms that also act by conjugation, like the diagonal automorphism and the graph automorphism for $D_l$ type. Then we argue what is the hardness of our security assumptions. 

Let $\phi$ be an automorphism of one of the classical Chevalley groups $G$: $\text{SL}(l+1,q), \text{O}(2l+1,q), \text{Sp}(2l,q),$ or $\text{O}(2l,q)$ of $A_l, B_l, C_l$ or $D_l$ type respectively. 
The automorphisms of these groups are described in Section~\ref{automorphismclassical}. From Theorem~\ref{autoch} we know that $\phi=c_{\chi}\iota\delta \gamma \theta$ where $c_{\chi}$ is a central automorphism, $\iota$ is an inner automorphism, $\delta$ is a diagonal automorphism, $\gamma$ is a graph automorphism and $\theta$ is a field automorphism.

The group of central automorphisms are too small and the field automorphisms reduce to a discrete logarithm in the field $\mathbb{F}_q$. So there is no benefit of using these in a MOR cryptosystem. Also there are not many graph automorphisms in classical Chevalley groups other than the $A_l$ and $D_l$ case. In the $D_l$ case these automorphisms act by conjugation. Recall here that, our automorphisms are presented as action on generators. It is clear~\cite[Section 7]{Ma} that if we can recover the conjugating matrix from the action on generators, then the security is $\mathbb{F}_{q^d}$, if not then the security is $\mathbb{F}_{q^{d^2}}$.

So from these we conclude that for a secure MOR cryptosystem we must look at automorphisms that act by conjugation, like the inner automorphisms. Inner automorphisms form a normal subgroup of $\Aut(G)$ and usually constitute the bulk of automorphisms. If $\phi$ is an inner automorphism, say $\iota_g\colon x\mapsto gxg^{-1}$, we would like to determine the conjugating element $g$. For $A_l$, the special linear group, it was done in~\cite{Ma}. We will follow the steps there for the present situation too. However, before we do that, let us digress briefly to observe that $G\rightarrow \text{Inn}(G)$ 
given by $g\mapsto \iota_g$ is a surjective group homomorphism. Thus if $G$ is generated by $g_1,g_2,\ldots, g_s$ then $\text{Inn}(G)$ is
generated by $\iota_{g_1},\ldots, \iota_{g_s}$. Let $\phi\in \text{Inn}(G)$. If we can find $g_j, j=1,2,\ldots,r$, generators,
such that $\phi=\prod\limits_{j=1}^r \iota_{g_j}$ then $\phi =\iota_g$ where $g=\prod\limits_{j=1}^r g_j$. This implies that our problem is 
equivalent to solving the word problem in $\text{Inn}(G)$. Note that solving word problem depends on how the group is represented and it is not invariant under group homomorphisms. Thus the algorithm described earlier to solve the word problem in the classical Chevalley groups does not help us in the present case.

\subsection{Reduction of security}

In this subsection, we show that for $A_l$ and $C_l$ case, the security of the MOR cryptosystem is the hardness of the discrete logarithm problem in $\mathbb{F}_{q^d}$. This is the same as saying that we can find the conjugating matrix up to a scalar multiple. We further show that the method that works for $A_l$ and $C_l$ does not work for $B_l$ and $D_l$. Let $\phi$ be an automorphism that works by conjugation, i.e., $\phi=\iota_g$ for some $g$ and we try to determine $g$.

{\bf Step 1:}  The automorphism $\phi$ is presented as action on generators $x_r(t)=I+te_r$ except $CG4$ in $B_l$ type. Thus $\phi(x_r(t))=g(I+te_r)g^{-1}=I+tge_rg^{-1}$ where $r\in \Phi$. This implies 
that we know $ge_rg^{-1}$ for all $r\in \Phi$. We first claim that we can determine $N:=gD$ where $D$ is sparse, in fact, diagonal in the case of $A_l$ and $C_l$ type.

In the case of $A_l$, write $g=[G_1,\ldots,G_i,\ldots,G_{l+1}]$, where $G_i$ are column vectors of $g$. Then $ge_{i,j}=\left[G_1,\ldots,G_{l+1}\right]e_{i,j}=[0,\ldots,0,G_i,0\ldots,0]$ where $G_i$ is at the $j\textsuperscript{th}$ place.  Multiplying this with $g^{-1}$ on the right, i. e.,  computing $ge_{i,j}g^{-1}$ determines $G_i$ up to a scalar multiple, say $d_i$. Thus, we know $N= gD$ where $D=\diag(d_1,\ldots,d_{l+1})$.

For the $C_l$ type we do the similar computation with the generators $I+te_{i,-i}$ and $I+te_{-i,i}$. 
Write $g$ in the column form as $\left[G_1,\ldots G_l, G_{-1},\ldots, G_{-l}\right]$. 
Now, 
\begin{enumerate}
\item $\left[G_1,\ldots G_l, G_{-1},\ldots, G_{-l}\right]e_{i,-i}=\left[0,\ldots,0,G_{i},0,\ldots,0\right] $
where $G_{i}$ is at $-i\textsuperscript{th}$ place. Multiplying this further with $g^{-1}$ gives us scalar multiple of $G_i$, say $d_i$.
\item $\left[G_1,\ldots G_l, G_{-1},\ldots, G_{-l}\right]e_{-i,i}=\left[0,\ldots,0,G_{-i},0,\ldots,0\right] $
where $G_{-i}$ is at $i\textsuperscript{th}$ place. Multiplying this with $g^{-1}$ gives us scalar multiple of $G_{-i}$, say $d_{-i}$.
\end{enumerate}
Thus we get $N=gD$ where $D$ is a diagonal matrix $\diag(d_1,\ldots, d_l,d_{-1},\ldots,d_{-l})$.

For $D_l$ type, write $g=\left[G_1,\ldots G_l, G_{-1},\ldots, G_{-l}\right]$. Now computing $ge_rg^{-1}$ gives the following equations:
\begin{enumerate}
\item $\left[G_1,\ldots G_l, G_{-1},\ldots, G_{-l}\right](e_{i,j}-e_{-j,-i})g^{-1}=\left[0,\ldots,0,G_i,0\ldots,0,G_{-j},0,\ldots\right]g^{-1} $
where $G_{i}$ is at $j\textsuperscript{th}$ place and $G_{-j}$ is at $-i\textsuperscript{th}$ place. This gives us linear combination of the columns $G_i$ and $G_{-j}$.
\item $\left[G_1,\ldots G_l, G_{-1},\ldots, G_{-l}\right](e_{i,-j} - e_{j,-i})g^{-1}=\left[0,\ldots,0,G_i,0\ldots,0,G_{j},0,\ldots\right]g^{-1}$
where $G_i$ is at $-j\textsuperscript{th}$ place and $G_j$ is at $-i\textsuperscript{th}$ place. This will give us linear combination of the columns $G_i$ and $G_j$.
\item $\left[G_1,\ldots G_l, G_{-1},\ldots, G_{-l}\right](e_{-i,j}- e_{-j,i})g^{-1}=\left[0,\ldots,0,G_{-i},0\ldots,0,G_{-j},0,\ldots\right]g^{-1}$
where $G_{-i}$ is at $j\textsuperscript{th}$ place and $G_{-j}$ is at $i\textsuperscript{th}$ place. This will give us linear combination of the columns $G_{-i}$ and $G_{-j}$.
\end{enumerate}
Thus we get $N=gD$ where $D$ is of the form $\begin{pmatrix}  W&X\\Y&Z \end{pmatrix}$
with $W$ a diagonal matrix, $Y$ anti-diagonal, $X$ has first column nonzero and $Z$ has the last column nonzero. This is not a diagonal matrix. 
One can do a similar computation for $B_l$ type.

{\bf Step 2:}
Now we compute $N^{-1}\phi(x_r(t))N=D^{-1}g^{-1}(gx_r(t)g^{-1})gD=I+D^{-1}e_rD$ which is equivalent to computing $D^{-1}e_rD$ for $r\in \Phi$. 

In the case of $A_l$ we have $D$ diagonal.
Thus by computing $D^{-1}e_{i,j}D$ we determine $d_i^{-1}d_j$ for $i\neq j$ and form a matrix $\diag(1,d_2^{-1}d_{1},\ldots,d_l^{-1}d_1)$ and multiply this to $N$ we get $d_1g$. Hence we can determine $g$ up to a scalar matrix.

In the $C_l$ case we can do similar computation as $D$ is diagonal. First compute $D^{-1}(e_{i,j}-e_{-j,-i})D$ to get $d_i^{-1}d_j$ and $d_{-i}^{-1}d_{-j}$ for $i\neq j$. 
Now compute $D^{-1}e_{i,-i}D, D^{-1}e_{-i,i}D$ to get $d_id_{-i}^{-1}, d_{-i}d_{i}^{-1}$. We form a matrix 
$$\diag(1,d_2^{-1}d_1,\ldots,d_l^{-1}d_1, d_{-1}^{-1}d_{-2}.d_{-2}^{-1}d_2.d_2^{-1}d_1,\ldots,d_{-l}^{-1}d_{-1}.d_{-1}^{-1}d_1)$$
and multiply it to $N=gD$ to get $d_1g$. Thus we can determine $g$ up to a scalar multiple and then the attack follows~\cite[Section 7.1.1]{Ma}. 

However in the case of $B_l$ and $D_l$ the matrix $D$ is not a diagonal matrix and the above method to determine $g$ does not work.
\section{Conclusion}
This section is similar to~\cite[Section 8]{Ma}. An useful public-key cryptosystem is a delicate dance between speed and the security. So one must talk about speed along with security. As we said in the introduction, this study was to find the embedding degree for the symplectic and orthogonal groups over finite fields of odd characteristic. So we will be somewhat brief with implementation details.

The implementation that we have in mind uses the row-column operations. Let $\langle g_1,g_2,\ldots,g_s\rangle$ be a set of generators for the orthogonal or symplectic group as described before. As is the custom with a MOR cryptosystem, the automorphisms $\phi$ and $\phi^m$ are presented as action on generators, i.e., we have $\phi(g_i)$ and $\phi^m(g_i)$ as matrices for $i=1,2,\ldots,s$.

To encrypt a message in this MOR cryptosystem, we compute $\phi^r$. We do that by \emph{square-and-multiply} algorithm. For this implementation, squaring and multiplying is almost the same. So we will refer to both squaring and multiplication as multiplication. Note that multiplication is composing of automorphisms.

The implementation that we describe in this paper, can work in parallel. Each instance computes $\pi^r(g_i)$ for $i=1,2,\ldots,s$. First thing that we do is write the matrix of $\phi(g_i)$ as a word in generators. So essentially the map $\phi$ becomes a map $g_i\mapsto w_i$ where $w_i$ is a word in generators of some fixed length. Then multiplication becomes essentially a replacement, replace all instances of $g_i$ by $w_i$. This can be done very fast. However, the length of the replaced word can become very large. The obvious question is, how soon are we going to write this word as a matrix. This is a difficult question to answer at this stage and depends on available computational resources.

Once we decide how often we change back to matrices, how are we going to change back to matrices? There can be a fairly easy \emph{time-memory} trade-offs.  Write all words up to a fixed length and the corresponding matrix as a pre-computed table and use this table to compute the matrices. Once we have matrices, we can multiply them together to generate the final output.
If writing all words is impossible, due to resource constraint, write some of it in a table. There are also many obvious relations among the generators of these groups. One can just store and use them. The best strategy for an efficient implementation is yet to be determined. It is clear now that there are many interesting and novel choices. 

The benefits of this MOR cryptosystem are:
\begin{description}
\item This can be implemented in parallel easily.
\item This implementation doesn't depend on the size of the characteristic of the field. This is an important property in light of Joux's recent improvement of the index-calculus attacks~\cite{joux}.
\end{description}

There is one issue with this MOR cryptosystem, the key-size is large. For parameters and complexity analysis of this cryptosystem, we refer to~\cite[Section 8]{Ma}.
\subsection{Further Research} We conclude this paper with two open directions for further research.
\begin{description}
\item What is the most efficient strategy to implement the MOR cryptosystem on Orthogonal and Symplectic groups that we described earlier?
\item What is the security for the twisted groups?
\end{description}

\bibliographystyle{amsplain}
\bibliography{paper}

\providecommand{\bysame}{\leavevmode\hbox to3em{\hrulefill}\thinspace}
\providecommand{\MR}{\relax\ifhmode\unskip\space\fi MR }
\providecommand{\MRhref}[2]{%
  \href{http://www.ams.org/mathscinet-getitem?mr=#1}{#2}
}
\providecommand{\href}[2]{#2}
\begin{thebibliography}{10}

\bibitem{balu}
R.~Balasubramanian and N.~Koblitz, \emph{The improbability than an elliptic
  curve has subexponential discrete log problem under the
  {Menezes-Okamoto-Vanstone} algorithm}, Journal of Cryptology \textbf{11}
  (1998), no.~2, 141--145.

\bibitem{joux}
Razvan Barbulescu, Plerrick Gaudry, Antoine Joux, and Emmanuel Thome, \emph{A
  heuristic quasi-polynomial algorithm for discrete logarithm in finite fields
  of small characteristic}, Eurocrypt2014, 2014, pp.~1--16.

\bibitem{bu1}
E.~I. Bunina, \emph{Automorphisms of chevalley groups of type {B} over local
  rings with 1/2}, Fundam. Prikl. Mat. \textbf{15} (2009), no.~7, 3--46.

\bibitem{bu2}
\bysame, \emph{Automorphisms of chevalley groups of types ${A_l}$, ${D_l}$, and
  ${E_l}$ over local rings with 1/2}, Fundam. Prikl. Mat. \textbf{15} (2009),
  no.~2, 35--59.

\bibitem{ca}
Roger Carter, \emph{Simple groups of {Lie} type}, Pure and Applied Mathematics,
  vol.~28, John Wiley \& Sons, 1972.

\bibitem{ch}
C.~Chevalley, \emph{Sur certains groupes simples}, Tohoku Math. J. \textbf{7}
  (1955), no.~2, 14--66.

\bibitem{CMT}
Arjeh~M. Cohen, Scott~H. Murray, and D.~E. Taylor, \emph{Computing in groups of
  {Lie} type}, Mathematics of computation \textbf{73} (2003), no.~247,
  1477--1498.

\bibitem{di}
Jean Dieudonne, \emph{On the automorphisms of the classical groups. with a
  supplement by {Loo-Keng Hua}}, Memoirs of the American Mathematical Society,
  1951.

\bibitem{gr}
Larry~C. Grove, \emph{Classical groups and geometric algebra}, vol.~39,
  American Mathematical Society, Graduate Studies in Mathematics, 2002.

\bibitem{gkkl1}
R.~M. Guralnick, W.~M. Kantor, M.~Kassabov, and A.~Lubotzky,
  \emph{Presentations of finite simple groups: profinite and cohomological
  approaches}, Groups Geom. Dyn. \textbf{1} (2007), no.~4, 469--523.

\bibitem{gkkl2}
\bysame, \emph{Presentations of finite simple groups: a quantitative approach},
  J. Amer. Math. Soc. \textbf{21} (2008), no.~3, 711--774.

\bibitem{gkkl3}
\bysame, \emph{Presentations of finite simple groups: a computational
  approach}, J. Eur. Math. Soc. \textbf{13} (2011), no.~2, 391--458.

\bibitem{Intro}
Jeffrey Hoffstein, Jill Pipher, and Joseph~H. Silverman, \emph{An introduction
  to mathematical cryptography}, Springer, 2008.

\bibitem{joux1}
Antoine Joux, \emph{A new index calculus algorithm with complexity
  ${L(1/4+o(1))}$ in small characteristic}, SAC2013, 2013, pp.~355--379.

\bibitem{kmr}
Max-Albert Knus, Alexander Merkurjev, Markus Rost, and Jean-Pierre Tignol,
  \emph{The book of involutions ({E}nglish summary) with a preface in {French}
  by {J.~Tits}}, vol.~44, American Mathematical Society Colloquium
  Publications, 1998.

\bibitem{lo}
C.~R. Leedham-Green and E.~A. O'Brien, \emph{Constructive recognition of
  classical groups in odd characteristic}, J. Algebra \textbf{322} (2009),
  no.~3, 833--881.

\bibitem{Ma}
Ayan Mahalanobis, \emph{A simple generalization of the {ElGamal} cryptosystem
  to non-abelian groups {II}}, Communications in Algebra \textbf{40} (2012),
  no.~9, 3583--3596.

\bibitem{Ma1}
\bysame, \emph{The {MOR} cryptosystem and finite $p$-groups}, Contemporary
  Mathematics, American Mathematical Soiety, 2014, to appear.

\bibitem{crypto2001}
Seong-Hun Paeng, Kil-Chan Ha, Jae~Heon Kim, Seongtaek Chee, and Choonsik Park,
  \emph{New public key cryptosystem using finite non-abelian groups}, Crypto
  2001 (J.~Kilian, ed.), LNCS, vol. 2139, Springer-Verlag, 2001, pp.~470--485.

\bibitem{sel}
G.~B. Seligman, \emph{Modular {L}ie algebras}, Springer-Verlag, 1967.

\bibitem{joseph}
Joseph Silverman and Joe Suzuki, \emph{Elliptic curve discrete logarithms and
  the index calculus}, Asiacrypt'98 (K.~Ohra and D.~Pei, eds.), LNCS, vol.
  1514, 1998, pp.~110--125.

\bibitem{st1}
Robert Steinberg, \emph{Automorphisms of finite linear groups}, Canadian
  Journal of Mathematics \textbf{12} (1960), 606--615.

\bibitem{st2}
\bysame, \emph{Lectures on {C}hevalley groups. notes prepared by {John
  Faulkner} and {Robert Wilson}}, Yale University, 1968.

\bibitem{va}
Nikolai Vavilov, \emph{Structure of chevalley groups over commutative rings},
  Nonassociative algebras and related topics (Hiroshima, 1990), World Sci.
  Publ., River Edge, NJ, 1991, pp.~219--335.

\end{thebibliography}
\end{document}